\documentclass[12pt]{iopart}

\makeatletter
\expandafter\let\csname equation*\endcsname\relax
\expandafter\let\csname endequation*\endcsname\relax
\let\[\relax \let\]\relax
\makeatother
\usepackage{amsmath}

\usepackage{iopams}      

\usepackage{geometry}
\usepackage{amsthm}
\usepackage{mathtools}
\usepackage{booktabs}
\usepackage{enumitem}
\usepackage{graphicx}
\usepackage{subcaption}
\usepackage[colorlinks=true,linkcolor=blue,citecolor=blue,urlcolor=blue]{hyperref}
\usepackage[capitalize,nameinlink]{cleveref}

\eqnobysec

\theoremstyle{plain}
\newtheorem{theorem}{Theorem}[section]
\newtheorem{lemma}[theorem]{Lemma}
\newtheorem{proposition}[theorem]{Proposition}
\newtheorem{corollary}[theorem]{Corollary}
\newtheorem{definition}[theorem]{Definition}
\theoremstyle{remark}
\newtheorem{remark}[theorem]{Remark}
\newtheorem{assumption}[theorem]{Assumption}

\crefname{equation}{}{}
\crefname{assumption}{Assumption}{Assumptions}
\Crefname{assumption}{Assumption}{Assumptions}
\crefname{section}{section}{sections}
\Crefname{section}{Section}{Sections}
\crefname{subsection}{section}{sections}
\Crefname{subsection}{Section}{Sections}
\crefname{subsubsection}{section}{sections}
\Crefname{subsubsection}{Section}{Sections}

\newcommand{\R}{\mathbb{R}}
\newcommand{\abs}[1]{\left\lvert #1 \right\rvert}
\newcommand{\E}{\mathbb{E}}
\newcommand{\Prob}{\mathbb{P}}
\newcommand{\Om}{\Omega}
\newcommand{\norm}[1]{\left\lVert #1 \right\rVert}
\newcommand{\inner}[2]{\left\langle #1,\, #2 \right\rangle}
\newcommand{\xdag}{x^{\dagger}}
\newcommand{\ys}{y^{\sigma}}
\newcommand{\xsk}[1][k]{x^{\sigma}_{#1}}

\newcommand{\Lap}{(-\Delta)}
\DeclareMathOperator{\dom}{\mathcal{D}}

\DeclareMathOperator{\nul}{\mathcal{N}}

\begin{document}

\title[Iterated Lavrentiev regularization under random noise]{Iterated
Lavrentiev regularization under random noise: variational source conditions and
rates in expectation}

\author{Ankit Singh$^1$, Hemant Sharma$^2$ and Kiran Davuluri$^3$}

\address{$^1$ Department of Applied Sciences and Humanities, SVKM's College of
Engineering, Shirpur 425405, India}
\address{$^2$ Department of Mathematics, Tungal School of Basic and Applied
Sciences, Jamkhandi 587301, India}
\address{$^3$ Infiswift AI, Milpitas, California 95053, USA}

\ead{sharmahemant39@gmail.com}

\begin{abstract}
The convergence theory of iterated Lavrentiev regularization for nonlinear
monotone ill-posed equations is deterministic: the data error is bounded by a known level that the
discrepancy principle uses explicitly. Under Gaussian noise no such bound exists. For a discretized two-dimensional
nonlinear inverse source problem for the Poisson equation, we analyse convergence in a finite-energy Gaussian noise model, keeping the
linearized and nonlinear cases apart. For the linearized iteration the noisy and
noise-free iterates differ by one resolvent applied to the noise; this gives an exact singular-value variance formula, a priori rates in
expectation, and the main result: for a family of discretizations whose distance functions are bounded uniformly
in the mesh, the rate is governed by that decay, which relaxes the classical
range condition. The rate degrades gracefully, down to
logarithmic, as the underlying continuum error loses smoothness, and a balancing
rule attains it without knowing the smoothness. For the full nonlinear iteration
under the order-one source condition, a truncated discrepancy principle fed only
the noise level attains $\sigma_E^{1/2}$ in expectation, in the energy noise
level $\sigma_E\asymp\sigma/h$, the rare event of large noise being controlled by
concentration. The rough nonlinear regime and the general-exponent nonlinear
rate are left open and supported numerically for kink and jump errors. We separate the deterministic worst-case, Gaussian, and white-noise settings,
prove a worst-case lower bound, and describe where the method is statistically
inefficient relative to spectral cut-off. Heterogeneous media and a cost
comparison with an adjoint-based scheme are given as supplementary material.
\end{abstract}

\vspace{2pc}
\noindent{\it Keywords}: Lavrentiev regularization, statistical inverse problems, random noise, rates in expectation, balancing principle, monotone operators

\vspace{1pc}
\noindent{\it Mathematics Subject Classification}: 65J20, 62G05, 65J15, 47J06, 35R30

\maketitle

\section{Introduction}\label{sec:intro}

Let $X$ be a real Hilbert space and consider the nonlinear ill-posed equation
\begin{equation}\label{eq:opeq}
  F(x)=y,\qquad F:\dom(F)\subseteq X\to X,
\end{equation}
with $F$ monotone, that is, $\inner{F(x)-F(w)}{x-w}\ge0$. For such equations the
iterated Lavrentiev method \cite{Tautenhahn2002,MahaleSingh2022}
\begin{equation}\label{eq:method}
  \xsk[k+1]=\xsk-\bigl(A^{\sigma}_{k}+\alpha_{k}I\bigr)^{-1}
  \bigl[F(\xsk)-\ys+\alpha_{k}(\xsk-x_{0})\bigr],
  \qquad A^{\sigma}_{k}:=F'(\xsk),
\end{equation}
with a decreasing parameter sequence $(\alpha_k)$ satisfying
\begin{equation}\label{eq:alpha}
  \lim_{k\to\infty}\alpha_k=0,\qquad 1\le\frac{\alpha_k}{\alpha_{k+1}}\le\ell\ (\ell>1),
\end{equation}
is attractive because it does not need the adjoint of $F'$. Its deterministic
theory was developed under a general source condition by Mahale and Nair
\cite{MahaleNair2009}, and under a weaker Lipschitz condition on $F'$ by Mahale and
Singh \cite{MahaleSingh2022}. In that theory the data error satisfies a
deterministic bound,
\begin{equation}\label{eq:detnoise}
  \norm{y-\ys}\le\delta,
\end{equation}
and the iteration is stopped by the discrepancy principle. The stopping index
$k_\delta$ is the first index with
\begin{equation}\label{eq:detstop}
  \norm{\alpha_{k}\bigl(A^{\sigma}_{k}+\alpha_{k}I\bigr)^{-1}\bigl(F(\xsk)-\ys\bigr)}
  \le\tau\delta,\qquad\tau>1 ,
\end{equation}
which uses $\delta$ explicitly. Under a H\"older source condition of order one the
resulting rate is $\norm{\xsk[k_\delta]-\xdag}=\mathcal{O}(\delta^{1/2})$
\cite[Thm.~3.5]{MahaleSingh2022}.

\paragraph{The statistical setting} In practice the measurement error is rarely a
vector with a known bounded norm. It is random. The natural model replaces
\eqref{eq:detnoise} by
\begin{equation}\label{eq:statnoise}
  \ys=F(\xdag)+e,\qquad e\ \text{a centred Gaussian noise of level}\ \sigma,
  \qquad \E\norm{e}^2=\sigma^2 .
\end{equation}
Now $\norm{y-\ys}=\norm{e}$ is random, and for Gaussian $e$ one has
$\Prob(\norm{e}>\delta)>0$ for every $\delta$. So no deterministic bound
\eqref{eq:detnoise} exists, the discrepancy stop \eqref{eq:detstop} has no
admissible $\delta$, and the error estimate of \cite{MahaleSingh2022} holds only
on the event $\{\norm{e}\le\delta\}$. A naive pathwise use of it can make the
expected error arbitrarily large, because rare realizations with
$\norm{e}\gg\sigma$ are amplified by the inverse problem.

A probabilistic analysis is therefore needed, one that controls the variance of
the propagated noise and the concentration of $\norm{e}$ instead of assuming a
worst-case bound. The question becomes how $\E\norm{\xsk[k_\sigma]-\xdag}$
behaves as $\sigma\to0$, and which stopping rules achieve the best decay. This
is the viewpoint of statistical inverse problems. For linear problems it is well
developed: minimax rates and oracle inequalities for spectral and Tikhonov
estimators in the white-noise model \cite{BissantzHohageMunkRuymgaart2007,
Cavalier2011,CavalierGolubevPicardTsybakov2002}, adaptive parameter choice by
balancing and Lepskii's method \cite{MathePereverzev2003,Lepskii1990}, minimax
theory of nonparametric estimation \cite{Tsybakov2009}, and the Bayesian view
with its frequentist contraction rates \cite{Stuart2010,KnapikVDVVZ2011}. For
nonlinear problems the theory is far less complete \cite{Nickl2023}, and
iterated Lavrentiev schemes, attractive because they avoid the adjoint, have to
our knowledge not been analysed under Gaussian noise with a stopping rule that
uses only the noise level.

\subsection*{Main mathematical contributions}

We analyse the iterated Lavrentiev method on a discretization of the
two-dimensional nonlinear inverse source problem for $-\Delta u=g(x)$
(\cref{sec:problem}), whose forward operator is monotone in an energy inner
product and satisfies the Lipschitz framework of \cite{MahaleSingh2022}. The
linearized and nonlinear analyses are kept apart throughout.

\begin{enumerate}[label=\textbf{C\arabic*.},leftmargin=3.0em,itemsep=3pt]
\item \textit{Linearized stochastic theory.} For the iteration linearized about
the exact solution, the noisy and noise-free iterates differ by one resolvent
applied to the noise (\cref{lem:biasvar}(a)). This gives an exact variance
formula in resolvent singular values (\cref{lem:variance}), a~priori rates
(\cref{thm:apriori}), and the main result (\cref{thm:vsc}): under a mesh-uniform
approximate source condition (\cref{ass:usc}), the rate is governed by the decay
of a distance function, which controls the bias from both sides
(\crefrange{lem:biasvsc}{lem:biasconv}). The rate degrades gracefully, down to
logarithmic, as the continuum error loses smoothness, and a balancing rule
attains it without knowing the smoothness (\cref{thm:balancing}).

\item \textit{Nonlinear theory under the order-one source condition.} We combine
the deterministic theory of \cite{MahaleSingh2022}, applied in the energy norm,
with a concentration bound for $\norm{e}_E$ (\cref{lem:conc}) and a computable
truncation of the estimator. The discrepancy principle, fed only a noise level,
then attains the rate $\sigma_E^{1/2}$ in expectation, with energy noise level
$\sigma_E\asymp\sigma/h$ (\cref{thm:discrepancy,prop:nonlinear}). The general
rate $\sigma^{\nu/(\nu+1)}$ is proved for the linearized iteration only; the
nonlinear case $\nu<1$ is open (\cref{sec:nonlinear}).

\item \textit{Rough regime, optimality, and validation.} We show where the
nonlinear analysis fails to close for rough errors (\cref{sec:nonlinear}),
separate the worst-case, finite-energy Gaussian, and white-noise settings with a
worst-case lower bound (\cref{sec:optimality}), and validate the theory
numerically (\cref{sec:numerics}); further experiments are in the supplementary
material.
\end{enumerate}

\paragraph{Relation to earlier work} Approximate and variational source
conditions for Lavrentiev regularization of nonlinear monotone equations were
developed by Hofmann, Kaltenbacher, and Resmerita
\cite{HofmannKaltenbacherResmerita2016}, who note that discrepancy-principle
rates under them are delicate. A general principle for lifting deterministic
rates to stochastic noise is due to Gerth, Hofinger, and Ramlau
\cite{GerthHofingerRamlau2017}. The deterministic rates for the iterated method
are due to Mahale and Nair \cite{MahaleNair2009} and Mahale and Singh
\cite{MahaleSingh2022}, and we use them in \cref{thm:discrepancy,prop:nonlinear}.
What is new here is the combination: a finite-energy Gaussian model, under which
a discrepancy rule fed only the noise level makes sense; the exact linearized
variance; a mesh-uniform, computable approximate source condition for the
discretized Poisson operator; the use of the exact cancellation of
\cref{lem:biasspec} in place of a general lifting principle; and the interaction
of noise, nonlinearity, and discretization, including a mesh-noise coupling
result.

\paragraph{Organization} \Cref{sec:setup} sets up the discretized problem and
the stochastic framework, \cref{sec:rates} proves the baseline rates,
\cref{sec:vsc} proves the main result and treats the nonlinear iteration, and
\cref{sec:validation} discusses optimality and numerics. The appendices give a
sharp variance bound, the balancing rule, and a concentration bound; the
supplementary material gives further validation and extensions to heterogeneous
media.

\section{Problem, method, and the stochastic framework}\label{sec:setup}

\subsection{Inverse source problem and deterministic facts}\label{sec:problem}

\paragraph{Inverse source problem} Let $\Om=(0,1)^2$ and let $g:\R\to\R$ be
$C^1$, non-decreasing, with $g(0)=0$. We use the model $g(s)=s^3$. Let $\Lap$
be the Dirichlet Laplacian and $K:=\Lap^{-1}$ its inverse, which is positive,
self-adjoint, and compact on $L^2(\Om)$, with $\norm{K}=\lambda_1^{-1}$ and
$\lambda_1=2\pi^2$. The forward operator of the source problem $-\Delta u=g(x)$,
$u|_{\partial\Om}=0$, is
\begin{equation}\label{eq:F}
  F(x)=K\,g(x)=\Lap^{-1}g(x).
\end{equation}
We seek the source $\xdag$ from a measurement of the state $y=F(\xdag)$.

\paragraph{Standing framework: the discretized problem and two inner products}
We work with the discretized problem, because the choice of Hilbert structure
matters. Discretize $\Om$ by the five-point Laplacian on an $N\times N$ interior
grid with mesh width $h=1/(N+1)$ and $n=N^2$ unknowns (\cref{sec:numerics}).
This gives a symmetric positive definite matrix $S\approx\Lap$, and we set
$K:=S^{-1}$ and $F(x):=K\,g(x)$ on $X:=\R^n$. On $X$ we use two inner products.
The first is the discrete $L^2$ inner product $\inner{u}{w}:=h^2\sum_iu_iw_i$
with norm $\norm{\cdot}$. All noise levels and all error rates are measured in
this norm. The second is the energy inner product
\begin{equation}\label{eq:energy}
  \inner{u}{w}_E:=\inner{u}{Sw},\qquad \norm{u}_E^2=\inner{u}{Su},
\end{equation}
which is a genuine inner product on $\R^n$. In the continuum the analogous
$\inner{u}{-\Delta w}$ is the $H^1_0$ inner product, not equivalent to the $L^2$
one; on $\R^n$ the two are equivalent with constants depending on $n$, uniformly
in one direction: the smallest eigenvalue of $S$ satisfies
$\lambda_1(S)\ge c_0>0$ with $c_0$ independent of $n$ (it converges to $2\pi^2$),
so
\begin{equation}\label{eq:normcomp}
  \norm{u}\le c_0^{-1/2}\,\norm{u}_E\qquad\text{for all }u\in\R^n .
\end{equation}
The reverse inequality carries the factor $\lambda_{\max}(S)^{1/2}\asymp h^{-1}$,
and we never use it in a rate.

The point of the energy inner product is that the monotone theory of
\cite{MahaleSingh2022} applies in the Hilbert space
$(\R^n,\inner{\cdot}{\cdot}_E)$. Indeed, for all $x,w$,
\[
  \inner{F(x)-F(w)}{x-w}_E=\inner{K(g(x)-g(w))}{S(x-w)}=\inner{g(x)-g(w)}{x-w}\ge0,
\]
because $g$ is non-decreasing. In the discrete $L^2$ inner product $F$ is not
monotone in general, since $K$ and the nonlinearity do not commute. We
therefore apply the deterministic theory in the energy space and transfer its
conclusions to the $L^2$ norm through \eqref{eq:normcomp}, which costs a
constant independent of $n$. The continuum problem \eqref{eq:F} provides the
motivation and the asymptotics of the spectrum (\cref{lem:weyl}); we do not
claim an infinite-dimensional theory.

\paragraph{Deterministic facts} The following properties of the discrete
forward operator are obtained by direct computation on the ball
$B_r^{\infty}(\xdag)=\{x:\norm{x-\xdag}_\infty<r\}$, with
$\rho:=\norm{\xdag}_\infty+r$. Throughout, $M_m$ denotes multiplication by the grid
function $m$, that is, the diagonal matrix $\mathrm{diag}(m)$.

\begin{proposition}[deterministic facts]\label{prop:det}
For $x\in B_r^{\infty}(\xdag)$ one has $F'(x)=K\,M_{g'(x)}$. Moreover
\begin{enumerate}[label=\textup{(\alph*)},leftmargin=2.2em,itemsep=1pt]
\item $F'(x)$ is self-adjoint and positive semidefinite in the energy inner
product \eqref{eq:energy}, since $\inner{F'(x)u}{w}_E=\inner{M_{g'(x)}u}{w}$;
its eigenvalues are real and non-negative. Hence $A^{\sigma}_k+\alpha_k I$ is
invertible, with
\[
\norm{(A^{\sigma}_k+\alpha_k I)^{-1}}_E\le\alpha_k^{-1},\qquad
\norm{\alpha_k(A^{\sigma}_k+\alpha_k I)^{-1}}_E\le1;
\]
\item \textup{(Lipschitz)} $\norm{F'(x)-F'(w)}\le L\norm{x-w}_\infty$,
$L=6\rho/\lambda_1$;
\item \textup{(scaling)} $\norm{F'(x)}\le a_0:=3\rho^2/\lambda_1$.
\end{enumerate}
Here, on the grid, $\lambda_1:=\lambda_{\min}(S)=\norm{K}^{-1}$, which lies slightly
below its continuum value $2\pi^2$ (for $N=49$, $19.733$ against $19.739$).
\end{proposition}

Throughout we write $A:=F'(\xdag)=KM$ with $M:=M_{g'(\xdag)}$, and
$e_0:=x_0-\xdag$. Several spectral statements need the multiplier to be
bounded away from zero.

\begin{assumption}[positivity of the multiplier]\label{ass:M}
There are constants $0<m_-\le m_+$ with $m_-\le g'(\xdag(\xi))\le m_+$ for all
grid points $\xi$. We write $\kappa_M:=(m_+/m_-)^{1/2}$.
\end{assumption}

For $g(s)=s^3$ this means that $\xdag$ is bounded away from zero on the grid.
It holds in the experiments of \cref{sec:vscnum} and of the supplementary
material, which use the bounded-away source
$1+\tfrac12\sin(\pi\xi_1)\sin(\pi\xi_2)$ ($\kappa_M\approx1.4$ to $1.5$), but for the
main test source
$\sin(\pi\xi_1)\sin(\pi\xi_2)$ of \cref{sec:numerics} the multiplier vanishes on
the boundary, so on the grid it holds only with a mesh-dependent constant. Under \cref{ass:M} the operator $A$
is similar to a symmetric matrix through a well-conditioned transformation:
\begin{equation}\label{eq:similar}
  A=KM=M^{-1/2}\,B\,M^{1/2},\qquad B:=M^{1/2}KM^{1/2},
\end{equation}
with $B$ symmetric positive definite in the discrete $L^2$ inner product and
$\norm{M^{1/2}}\,\norm{M^{-1/2}}\le\kappa_M$. Consequently $A$ and $B$ share
their eigenvalues $\lambda_1\ge\cdots\ge\lambda_n>0$, fractional powers are
defined by $A^{\nu}:=M^{-1/2}B^{\nu}M^{1/2}$, and the resolvent is bounded in
the $L^2$ norm uniformly in $n$:
\begin{equation}\label{eq:resolventL2}
  \norm{(A+\alpha I)^{-1}}\le\kappa_M\,\alpha^{-1},\qquad
  \norm{\alpha(A+\alpha I)^{-1}}\le\kappa_M ,\qquad \alpha>0 .
\end{equation}
Without \cref{ass:M} we only have the energy-norm bounds of
\cref{prop:det}(a), and the passage to the $L^2$ norm goes through
\eqref{eq:normcomp}.

We impose the H\"older source condition of order one,
\begin{equation}\label{eq:source}
  e_0=A v,\qquad v\in\nul(A)^{\perp},
\end{equation}
together with the smallness conditions of \cite{MahaleSingh2022}. These
guarantee that the iterates are well defined and stay in $B_r^{\infty}(\xdag)$
as long as the data error is bounded by the level at which the theory is
applied. We keep the range condition \eqref{eq:source} through \cref{sec:rates}
and relax it in \cref{sec:vsc}.

\paragraph{The role of the sup norm} Noise and rates are measured in
$\norm{\cdot}$, and the deterministic theory runs in $\norm{\cdot}_E$. The
nonlinearity conditions, however, are $L^\infty$ statements: the Lipschitz bound
of \cref{prop:det}(b) uses $\norm{x-w}_\infty$, and well-posedness holds on
$B_r^{\infty}(\xdag)$. Under the smallness conditions the iterates stay in that
ball, so every nonlinear error term carries one $L^\infty$ factor; a Taylor
remainder, for instance, is bounded by
$\tfrac{L}{2}\norm{e_k}\,\norm{e_k}_\infty\le\tfrac{Lr}{2}\norm{e_k}$. Since
$\norm{\cdot}\le\norm{\cdot}_\infty$ on the grid while the reverse inequality does
not hold uniformly in $n$, this factor is small only in the classical regime. This is why the
nonlinear closure of \cref{sec:nonlinear} works there and breaks in the rough
regime.

We also use the noise-free iterates $x_k$, obtained from \eqref{eq:method} with
exact data $y$ in place of $\ys$, and the corresponding deterministic bias
\begin{equation}\label{eq:bias}
  b_k:=\norm{x_k-\xdag}.
\end{equation}
Under \eqref{eq:source}, the noise-free analysis of \cite{MahaleSingh2022}
applied in the energy space, followed by \eqref{eq:normcomp}, gives a constant
$C_b>0$ with
\begin{equation}\label{eq:biasbound}
  b_k\le C_b\,\alpha_k\,\norm{v},\qquad k\ge0 .
\end{equation}
For the linearized iteration the same bound follows directly from
\cref{lem:biasspec}, with $C_b=\kappa_M\ell$ under \cref{ass:M}.

\paragraph{The linearized iteration} By the linearized iteration we mean
\eqref{eq:method} with the fixed operator $A=F'(\xdag)$ in place of
$A^{\sigma}_k=F'(\xsk)$ and with $F(\xsk)-\ys$ replaced by $A(\xsk-\xdag)-e$.
For it the stochastic analysis is exact and the bias is a single resolvent step
(\cref{lem:biasspec}). The nonlinear iteration is treated on a high-probability
event, with the estimator truncated outside it (\cref{thm:discrepancy,prop:nonlinear}).

\paragraph{Statistical noise model} We make \eqref{eq:statnoise} precise. On
the $n$-dimensional discretization the noise is
\begin{equation}\label{eq:noisegen}
  e=\frac{\sigma}{h\sqrt n}\,\xi,\qquad \xi\sim\mathcal N(0,I_n),
\end{equation}
so that $e$ is an isotropic centred Gaussian vector with $\E\norm{e}^2=\sigma^2$
in the discrete $L^2$ norm and covariance
$\E\inner{e}{u}\inner{e}{w}=(\sigma^2/n)\inner{u}{w}$. This is a
finite-energy model: the variance per orthonormal mode is $\sigma^2/n$, and the
total noise energy is $\sigma^2$ whatever the mesh. It differs
from the white-noise model $\ys=F(\xdag)+\sigma\dot W$ of
\cite{BissantzHohageMunkRuymgaart2007}, in which the variance per mode is
$\sigma^2$ and the total energy grows with $n$; \cref{sec:optimality} compares
the two.

The linearized theory uses $\norm{e}$, which concentrates about $\sigma$
(\cref{lem:conc}). The nonlinear theory, applied in the energy norm, uses
$\norm{e}_E$, for which
\begin{equation}\label{eq:enoise}
  \E\norm{e}_E^2=\frac{\sigma^2}{n}\tr(S)=\sigma^2\,\overline{\lambda}(S),
  \qquad \overline{\lambda}(S):=\tfrac1n\tr(S)\asymp h^{-2},
\end{equation}
so the energy noise level is $\sigma_E:=\sigma\,\overline{\lambda}(S)^{1/2}\asymp\sigma/h$.
This mesh factor is genuine: it is why the nonlinear result needs a mesh-noise
relation to give a rate in $\sigma$, whereas the linearized result does not.

\subsection{Bias-variance decomposition and the variance bound}\label{sec:biasvar}

The error of the noisy iterate splits into the deterministic bias and a
stochastic term that measures how the method propagates the noise. For the
linearized iteration the split is exact. For the nonlinear iteration it holds
on the event where the iterates stay in the ball, at the price of a stability
constant and a maximum over the previous steps.

\begin{lemma}[Bias-variance decomposition]\label{lem:biasvar}
\begin{enumerate}[label=\textup{(\alph*)},leftmargin=2.2em,itemsep=2pt]
\item \textup{(Linearized iteration.)} Let $\xsk$ and $x_k$ be the noisy and the
noise-free iterates of the linearized iteration, both started at $x_0$. Then,
for every $k\ge0$,
\begin{equation}\label{eq:linsplit}
\begin{gathered}
  \xsk[k+1]-x_{k+1}=(A+\alpha_kI)^{-1}e,\quad\text{hence}\\
  \norm{\xsk[k+1]-\xdag}\le b_{k+1}+n_k,\quad n_k:=\norm{(A+\alpha_kI)^{-1}e}.
\end{gathered}
\end{equation}
\item \textup{(Nonlinear iteration.)} There is a constant $C_s>0$ \textup{(}the
stability constant of \cite[Lem.~5.1]{MahaleSingh2022}\textup{)} such that, on
the event that the iterates $\xsk[j]$, $j\le k$, are defined and lie in
$B_r^{\infty}(\xdag)$, and with all norms taken in the energy norm in which
\cite[Lem.~5.1]{MahaleSingh2022} is applied,
\begin{equation}\label{eq:biasvar}
  \norm{\xsk-\xdag}_E\le \underbrace{\norm{x_k-\xdag}_E}_{\text{bias}}
  \;+\;C_s\,\max_{j<k}\,\underbrace{\norm{(A^{\sigma}_j+\alpha_j I)^{-1}e}_E}_{=:n^{\sigma}_j} .
\end{equation}
\end{enumerate}
\end{lemma}

\begin{proof}
(a) With the linearization, the noisy iteration reads
$\xsk[k+1]=\xsk-(A+\alpha_kI)^{-1}\bigl[A(\xsk-\xdag)-e+\alpha_k(\xsk-x_0)\bigr]$,
and the noise-free one is the same with $e=0$. Write $e^\sigma_k:=\xsk-\xdag$.
Exactly as in the proof of \cref{lem:biasspec} below, the bracket equals
$(A+\alpha_kI)e^\sigma_k-\alpha_ke_0-e$, so that
\[
  e^\sigma_{k+1}=\alpha_k(A+\alpha_kI)^{-1}e_0+(A+\alpha_kI)^{-1}e .
\]
The first term is the noise-free error $e_{k+1}$ of \eqref{eq:biasrec}.
Subtracting gives the identity in \eqref{eq:linsplit}, and the triangle
inequality gives the bound.

(b) Throughout this part $\norm{\cdot}$ denotes the energy norm $\norm{\cdot}_E$.
Fix $k$ and split $\xsk-\xdag=(\xsk-x_k)+(x_k-\xdag)$, so that
$\norm{\xsk-\xdag}\le\norm{\xsk-x_k}+b_k$. Set $w_j:=\xsk[j]-x_j$, with $w_0=0$.
Subtracting the noise-free iteration from \eqref{eq:method} shows that $w_{j+1}$
is obtained from $w_j$ by a recursion in which the data enter in exactly one
place, through $(A^{\sigma}_j+\alpha_jI)^{-1}(\ys-y)=(A^{\sigma}_j+\alpha_jI)^{-1}e$.
All other terms involve $w_j$, the operators $F'(\xsk[j])$ and $F'(x_j)$, and
Taylor remainders, and on the event of the statement they are bounded in
\cite[Lem.~5.1]{MahaleSingh2022} through the Lipschitz condition of
\cref{prop:det}(b). The outcome of that lemma is the one-step contraction
$\norm{w_{j+1}}\le C_9\norm{w_j}+n^\sigma_j$ with $C_9<1$. The deterministic
proof bounds $n^\sigma_j$ by $\delta/\alpha_j$; we keep it as it is, because
under random noise this norm, not $\delta$, is what we can control. Unrolling
from $w_0=0$ gives
$\norm{w_k}\le\sum_{j<k}C_9^{\,k-1-j}n^\sigma_j\le(1-C_9)^{-1}\max_{j<k}n^\sigma_j$,
which is \eqref{eq:biasvar} with $C_s=(1-C_9)^{-1}$. We do not claim that the
$n^\sigma_j$ are monotone in $j$: the operators $A^\sigma_j$ change with $j$, so
only the maximum is available.
\end{proof}

For the propagated noise of the linearized iteration the second moment can be
computed exactly. The computation must be done in the inner product in which
the noise is isotropic, that is, the discrete $L^2$ one, in which $A$ is not
symmetric. This is why singular values, not eigenvalues, appear.

\begin{lemma}[Variance of the propagated noise; linearized iteration]\label{lem:variance}
Let $B_k:=(A+\alpha_kI)^{-1}$ and let $s_1(B_k)\ge\cdots\ge s_n(B_k)$ be its
singular values in the discrete $L^2$ inner product. Then
\begin{equation}\label{eq:var}
  \E\,n_k^2=\frac{\sigma^2}{n}\tr\bigl[B_k^{*}B_k\bigr]
  =\frac{\sigma^2}{n}\sum_{i=1}^{n}s_i(B_k)^2 .
\end{equation}
Under \cref{ass:M}, with $\{\lambda_i\}$ the eigenvalues of $A$ and
\begin{equation}\label{eq:phik}
  \phi_k^2:=\frac1n\sum_{i=1}^{n}\frac{1}{(\lambda_i+\alpha_k)^2}\;\le\;\frac{1}{\alpha_k^2},
\end{equation}
one has the two-sided bound
\begin{equation}\label{eq:var2}
  \kappa_M^{-2}\,\sigma^2\phi_k^2\;\le\;\E\,n_k^2\;\le\;\kappa_M^{2}\,\sigma^2\phi_k^2
  \;\le\;\kappa_M^{2}\,\frac{\sigma^2}{\alpha_k^2},
\end{equation}
with equality in the first two inequalities when $g'(\xdag)$ is constant. Without \cref{ass:M},
\eqref{eq:var} still holds, and $\E n_k^2\le c_0^{-1}\lambda_{\max}(S)\,
\sigma^2/\alpha_k^2$ by \cref{prop:det}(a) and \eqref{eq:normcomp}, a bound
whose constant grows like $h^{-2}$.
\end{lemma}

\begin{proof}
Write $n_k^2=\inner{B_ke}{B_ke}=\inner{e}{B_k^{*}B_ke}$, where $B_k^{*}$ is the
adjoint in the discrete $L^2$ inner product. Taking expectations and using
linearity of the trace,
\[
  \E\norm{B_ke}^2=\E\,\tr\bigl(B_k^{*}B_k\,ee^{*}\bigr)
  =\tr\bigl(B_k^{*}B_k\,\E[ee^{*}]\bigr)=\frac{\sigma^2}{n}\tr\bigl(B_k^{*}B_k\bigr),
\]
because $\E[ee^{*}]=\frac{\sigma^2}{n}I$ by \eqref{eq:noisegen}. The trace of
$B_k^{*}B_k$ is the sum of the squared singular values, which gives
\eqref{eq:var}. We stress that $A$ is not symmetric in this inner product, so
this trace is not in general $\sum_i(\lambda_i+\alpha_k)^{-2}$.

Under \cref{ass:M}, the similarity \eqref{eq:similar} gives
$B_k=M^{-1/2}(B+\alpha_kI)^{-1}M^{1/2}$ with $B$ symmetric, and the singular
values of $(B+\alpha_kI)^{-1}$ are exactly $(\lambda_i+\alpha_k)^{-1}$. For any
matrices, $s_i(XCY)\le\norm{X}\,s_i(C)\,\norm{Y}$, and applying this in both
directions with $X=M^{-1/2}$, $Y=M^{1/2}$ gives
$\kappa_M^{-1}(\lambda_i+\alpha_k)^{-1}\le s_i(B_k)\le\kappa_M(\lambda_i+\alpha_k)^{-1}$
for every $i$. Squaring, summing, and dividing by $n$ gives \eqref{eq:var2}.
The last inequality in \eqref{eq:var2} and \eqref{eq:phik} follow from
$\lambda_i\ge0$. When $g'(\xdag)$ is constant, $M$ is a multiple of the
identity, $\kappa_M=1$, and the first two inequalities are equalities. The bound without
\cref{ass:M} follows from $\norm{B_ku}\le c_0^{-1/2}\norm{B_ku}_E\le
c_0^{-1/2}\alpha_k^{-1}\norm{u}_E\le c_0^{-1/2}\lambda_{\max}(S)^{1/2}\alpha_k^{-1}\norm{u}$.
\end{proof}

The two competing effects are now explicit. The bias \eqref{eq:biasbound}
decreases like $\alpha_k$. The root mean square of the propagated noise,
$(\E n_k^2)^{1/2}\le\kappa_M\sigma/\alpha_k$, increases as $\alpha_k\downarrow0$.
Balancing the two governs the rate. In the experiments reported in the supplementary
material the
Monte Carlo variance matches $\sigma^2\phi_k^2$ to within one percent although
$\kappa_M=1.5$ there, so the two-sided bound \eqref{eq:var2} is conservative.

\begin{remark}[Linearization and the nonlinear remainder]\label{rem:linearization}
The identities \eqref{eq:linsplit} and \eqref{eq:var} hold for the linearized
iteration only. For the nonlinear iteration, where $A^{\sigma}_k=F'(\xsk)$ is
random and correlated with $e$, they become bounds. The operator replacement is controlled
by the Lipschitz condition of \cref{prop:det}(b) and the resolvent identity:
\[
  (A^{\sigma}_k+\alpha_kI)^{-1}-(A+\alpha_kI)^{-1}
  =(A^{\sigma}_k+\alpha_kI)^{-1}(A-A^{\sigma}_k)(A+\alpha_kI)^{-1},
\]
so that, in the energy norm,
\[
\norm{(A^{\sigma}_k+\alpha_kI)^{-1}-(A+\alpha_kI)^{-1}}_E\le\alpha_k^{-2}
\norm{A-A^{\sigma}_k}_E\lesssim\alpha_k^{-2}L\norm{\xsk-\xdag}_\infty .
\]
This is of
higher order than the resolvent itself as long as
$L\norm{\xsk-\xdag}_\infty\lesssim\alpha_k$, which is exactly the regime of the
classical source condition (\cref{sec:nonlinear}). Similarly, writing
$\xsk-\xdag=\eta_k+\zeta_k$ with $\eta_k$ the linearized error, the remainder
obeys
\[
  \norm{\zeta_{k+1}}\le \frac{L}{2}\,\norm{\xsk-\xdag}\,\norm{\xsk-\xdag}_\infty
  + \text{(propagated)} ,
\]
a term quadratic in the current error. For rough initial errors it is no longer
of higher order (\cref{sec:nonlinear,rem:nlopen}). The experiments use the full
nonlinear iteration \eqref{eq:method} throughout.
\end{remark}

\subsection{Spectral preliminaries: bias bound, decay, effective dimension}
\label{sec:spectral}

The crude bound $\phi_k^2\le\alpha_k^{-2}$ of \cref{lem:variance} is enough for
the rates below. A self-contained control of the bias, and the eigenvalue decay
and effective dimension that govern the attainable accuracy, follow from the
spectrum of $A=F'(\xdag)$. By \cref{prop:det}(a) this operator is self-adjoint
and positive in $\inner{\cdot}{\cdot}_E$. The sharper spectral form of the
variance bound is deferred to \cref{app:sharpvar}. Let $\{(\lambda_i,u_i)\}_{i\ge1}$
be the eigenpairs, $\lambda_1\ge\lambda_2\ge\cdots>0$, with $\{u_i\}$ an
$E$-orthonormal basis, so that $f(A)w=\sum_i f(\lambda_i)\inner{w}{u_i}_E u_i$
for Borel $f$.

\subsubsection{A self-contained bias bound}

\begin{lemma}[Bias of the noise-free iteration]\label{lem:biasspec}
For the linearization of \eqref{eq:method} about $\xdag$ with exact data, the
noise-free error $e_k:=x_k-\xdag$ satisfies
\begin{equation}\label{eq:biasrec}
  e_k=\alpha_{k-1}(A+\alpha_{k-1}I)^{-1}e_0,
\end{equation}
which is a single Lavrentiev step at parameter $\alpha_{k-1}$. The reason is that
the penalty in \eqref{eq:method} is anchored at the fixed $x_0$, so the linear
dependence on the previous iterate cancels and no product filter builds up. Under
\cref{ass:M} and the source condition $e_0=A^{\nu}v$, $0<\nu\le1$, with
$A^{\nu}$ defined through \eqref{eq:similar},
\begin{equation}\label{eq:biasspec}
  b_k=\norm{e_k}\le\kappa_M\Bigl(\sup_{\lambda\ge0}\frac{\alpha_{k-1}\lambda^{\nu}}
  {\lambda+\alpha_{k-1}}\Bigr)\norm{v}
  =\kappa_M\nu^{\nu}(1-\nu)^{1-\nu}\alpha_{k-1}^{\nu}\,\norm{v}\le\kappa_M\alpha_{k-1}^{\nu}\,\norm{v}.
\end{equation}
In the energy norm, where $A$ is self-adjoint, the same bound holds with
$\kappa_M=1$ and $\norm{v}_E$ in place of $\norm{v}$.
\end{lemma}

\begin{proof}
With exact data $y=F(\xdag)=K g(\xdag)$, linearizing $F$ about $\xdag$ replaces
$F(x_k)-y$ by $A(x_k-\xdag)$ and $F'(x_k)$ by $A$. The iteration
\eqref{eq:method} then reads
\[
  x_{k+1}=x_k-(A+\alpha_kI)^{-1}\bigl[A(x_k-\xdag)+\alpha_k(x_k-x_0)\bigr].
\]
Write $x_k-x_0=(x_k-\xdag)-(x_0-\xdag)=e_k-e_0$ and subtract $\xdag$ from both
sides. The bracket becomes $Ae_k+\alpha_k e_k-\alpha_k e_0=(A+\alpha_kI)e_k-\alpha_ke_0$,
and so
\[
  e_{k+1}=e_k-(A+\alpha_kI)^{-1}(A+\alpha_kI)e_k+\alpha_k(A+\alpha_kI)^{-1}e_0
  =\alpha_k(A+\alpha_kI)^{-1}e_0 .
\]
The term $e_k$ cancels exactly. This is \eqref{eq:biasrec} after shifting the
index by one. The cancellation is a consequence of the penalty being anchored at
the fixed point $x_0$ rather than at the previous iterate.

For the bound, insert the source condition $e_0=A^{\nu}v$ and use the
similarity \eqref{eq:similar}: $\alpha(A+\alpha I)^{-1}A^{\nu}=M^{-1/2}\bigl[\alpha(B+\alpha I)^{-1}B^{\nu}\bigr]M^{1/2}$,
where the bracket is a function of the symmetric matrix $B$ and has $L^2$ norm
equal to the supremum of that function over the spectrum. Together with
$\norm{M^{-1/2}}\norm{M^{1/2}}\le\kappa_M$,
\[
  b_k=\norm{\alpha_{k-1}(A+\alpha_{k-1}I)^{-1}A^{\nu}v}
  \le\kappa_M\sup_{\lambda\ge0}\frac{\alpha_{k-1}\lambda^{\nu}}{\lambda+\alpha_{k-1}}\;\norm{v}.
\]
To evaluate the supremum, set $h(\lambda)=\lambda^{\nu}/(\lambda+\alpha)$ with
$\alpha=\alpha_{k-1}$. For $0<\nu<1$ the derivative
$h'(\lambda)=\lambda^{\nu-1}\bigl(\nu(\lambda+\alpha)-\lambda\bigr)/(\lambda+\alpha)^2$
vanishes at $\lambda^{\ast}=\nu\alpha/(1-\nu)$, which is a maximum. Substituting,
$h(\lambda^{\ast})=\nu^{\nu}(1-\nu)^{1-\nu}\alpha^{\nu-1}$, so
$\alpha\,h(\lambda^{\ast})=\nu^{\nu}(1-\nu)^{1-\nu}\alpha^{\nu}$. Both factors
$\nu^{\nu}$ and $(1-\nu)^{1-\nu}$ lie in $(0,1]$, which gives the final
inequality in \eqref{eq:biasspec}. For $\nu=1$ the supremum of
$\alpha\lambda/(\lambda+\alpha)$ is $\alpha$, attained as $\lambda\to\infty$, and
the same conclusion holds.
\end{proof}

The single-step form \eqref{eq:biasrec} fixes the qualification of the linearized
scheme at $\nu=1$, in line with \cref{rem:nu}. The iteration in \eqref{eq:method}
matters only through the nonlinear remainder, treated in \cref{sec:nonlinear}.
This also makes \eqref{eq:biasbound} self-contained for the linearized scheme.

\subsubsection{Eigenvalue decay and effective dimension}

\begin{lemma}[Weyl decay]\label{lem:weyl}
Let \cref{ass:M} hold. Then there exist $0<c_-\le c_+$, depending on $m_\pm$
but not on $n$, with $c_-\,i^{-1}\le\lambda_i\le c_+\,i^{-1}$ for all
$1\le i\le n$.
\end{lemma}

\begin{proof}
By \eqref{eq:similar}, $A$ and $B=M^{1/2}KM^{1/2}$ share their eigenvalues, and
$B$ is symmetric positive definite in the discrete $L^2$ inner product, so its
eigenvalues are described by the min-max principle. For $u\in\R^n$ put
$v=M^{1/2}u$. Then
\[
  \frac{\inner{Bu}{u}}{\inner{u}{u}}=\frac{\inner{Kv}{v}}{\inner{M^{-1}v}{v}},
  \qquad
  m_+^{-1}\norm{v}^2\le\inner{M^{-1}v}{v}\le m_-^{-1}\norm{v}^2 .
\]
Since $u\mapsto M^{1/2}u$ maps $i$-dimensional subspaces onto $i$-dimensional
subspaces, the min-max characterization gives
$m_-\,\mu_i(K)\le\lambda_i(B)\le m_+\,\mu_i(K)$ for every $i$, where $\mu_i(K)$
are the eigenvalues of $K=S^{-1}$ in decreasing order.

It remains to recall the discrete Weyl law. The eigenvalues of the five-point
Laplacian on the $N\times N$ grid are
$\mu_{pq}(S)=\frac{4}{h^2}\bigl(\sin^2\tfrac{p\pi h}{2}+\sin^2\tfrac{q\pi h}{2}\bigr)$,
$1\le p,q\le N$. Since $\frac{2}{\pi}t\le\sin t\le t$ on $[0,\tfrac{\pi}{2}]$,
one has $4(p^2+q^2)\le\mu_{pq}(S)\le\pi^2(p^2+q^2)$, with constants independent
of $h$. Counting lattice points in a quarter disc, the $j$-th smallest value of
$p^2+q^2$ is $\asymp j$. Hence $\mu_j(S)\asymp j$ and $\mu_j(K)\asymp j^{-1}$
for $1\le j\le n$, uniformly in $n$. Combining the two displays gives the claim.
In the continuum the same argument, with Weyl's law
$\mu_j(\Lap)\sim(4\pi/|\Om|)j$ for the Dirichlet Laplacian on
$\Om\subset\R^2$, gives $\lambda_i\asymp i^{-1}$ for all $i$.
\end{proof}

\begin{definition}\label{def:neff}
The effective dimension at level $\alpha$ on the $n$-dimensional
discretization is
\[
  N_{\rm eff}(\alpha):=\tr\bigl[A(A+\alpha I)^{-1}\bigr]=\sum_{i=1}^{n}\frac{\lambda_i}
  {\lambda_i+\alpha}.
\]
It counts the modes that are effectively inverted at regularization level
$\alpha$. Under \cref{ass:M}, \cref{lem:weyl} gives
$\lambda_i/(\lambda_i+\alpha)\asymp(1+\alpha i)^{-1}$, and hence
\begin{equation}\label{eq:neff}
  N_{\rm eff}(\alpha)\asymp\sum_{i=1}^{n}\frac{1}{1+\alpha i}\asymp
  \alpha^{-1}\log(1+\alpha n).
\end{equation}
The dependence on $n$ is essential. With $\lambda_i\asymp i^{-1}$ the series
$\sum_i\lambda_i/(\lambda_i+\alpha)$ diverges, so $A(A+\alpha I)^{-1}$ is not
trace class in the continuum two-dimensional problem, and the effective
dimension is a property of the discretized problem. For $\alpha n\gg1$ it
behaves like $\alpha^{-1}\log(\alpha n)$; for $\alpha n\lesssim1$ it is
$\asymp n$, that is, all modes are inverted.
\end{definition}

\section{Rates in expectation}\label{sec:rates}

\subsection{A priori parameter choice}

\begin{proposition}[A priori rate in expectation; linearized iteration]\label{thm:apriori}
Let \cref{prop:det} and \cref{ass:M} hold, and let the source condition
\eqref{eq:source} hold. For the linearized iteration, choose the stopping index
$k_\sigma$ a~priori so that $\alpha_{k_\sigma-1}\asymp\sigma^{1/2}$. Then
\begin{equation}\label{eq:apriorirate}
  \Bigl(\E\,\norm{\xsk[k_\sigma]-\xdag}^2\Bigr)^{1/2}\le C\,\sigma^{1/2},
\end{equation}
with $C=\sqrt{2}\,\kappa_M(\norm{v}+1)$ up to the $\asymp$ constant.
\end{proposition}

\begin{proof}
By \cref{lem:biasvar}(a) with $k$ replaced by $k-1$, and the elementary
inequality $(a+b)^2\le2a^2+2b^2$,
\[
  \norm{\xsk-\xdag}^2\le 2b_k^2+2n_{k-1}^2 .
\]
Take expectations. The bias is deterministic, and \cref{lem:biasspec} with
$\nu=1$ gives $b_k\le\kappa_M\alpha_{k-1}\norm{v}$. The noise term is bounded
by \cref{lem:variance}: $\E n_{k-1}^2\le\kappa_M^2\sigma^2/\alpha_{k-1}^2$. Hence
\[
  \E\norm{\xsk-\xdag}^2\le 2\kappa_M^2\norm{v}^2\alpha_{k-1}^2
  +\frac{2\kappa_M^2\sigma^2}{\alpha_{k-1}^2}=:f(\alpha_{k-1}).
\]
The function $f(\alpha)=a\alpha^2+b\alpha^{-2}$ with $a=2\kappa_M^2\norm{v}^2$ and
$b=2\kappa_M^2\sigma^2$ is minimized where $f'(\alpha)=2a\alpha-2b\alpha^{-3}=0$,
that is, at $\alpha^4=b/a=\sigma^2/\norm{v}^2$, so at
$\alpha=\norm{v}^{-1/2}\sigma^{1/2}$. The minimizer scales like $\sigma^{1/2}$,
and any choice $\alpha\asymp\sigma^{1/2}$ changes the value of $f$ by at most a
constant factor. Taking for definiteness $\alpha=\sigma^{1/2}$ gives
$f(\sigma^{1/2})=2\kappa_M^2(\norm{v}^2+1)\sigma$. Taking square roots and using
$(p^2+q^2)^{1/2}\le p+q$ gives \eqref{eq:apriorirate} with
$C=\sqrt2\,\kappa_M(\norm{v}+1)$. The geometric schedule \eqref{eq:alpha}
contains an index with $\sigma^{1/2}/\ell\le\alpha_{k_\sigma-1}\le\sigma^{1/2}$,
and the factor $\ell$ is absorbed in the $\asymp$ constant.
\end{proof}

\begin{remark}[General smoothness]\label{rem:nu}
For a H\"older source condition of order $\nu\in(0,1]$, $e_0=A^\nu v$, the bias
obeys $b_k\lesssim\alpha_{k-1}^{\nu}$ by \cref{lem:biasspec}. The same balancing
gives $\alpha_{k_\sigma-1}\asymp\sigma^{1/(\nu+1)}$ and the rate
$(\E\norm{\xsk[k_\sigma]-\xdag}^2)^{1/2}=\mathcal{O}\bigl(\sigma^{\nu/(\nu+1)}\bigr)$.
The qualification of Lavrentiev regularization is $\nu=1$, which gives the best
rate $\sigma^{1/2}$ obtained above.
\end{remark}

\subsection{The discrepancy principle with the noise level}

The a~priori choice needs $\nu$; the discrepancy principle does not, but under
random noise it must be fed a noise level, since no pathwise bound
\eqref{eq:detnoise} exists. We show that this gives the rate $\sigma_E^{1/2}$ for
the full nonlinear iteration under the classical source condition, once the
estimator is truncated on the rare event of large noise. The key fact is that
the noise norm concentrates sharply in high dimension.

\begin{lemma}[Concentration of the noise norm]\label{lem:conc}
For the isotropic Gaussian $e$ of \eqref{eq:noisegen} and any $\beta>1$,
\begin{equation}\label{eq:conc}
  \Prob\bigl(\norm{e}>\beta\sigma\bigr)\le \exp\!\bigl(-c_\beta\,n\bigr),
  \qquad c_\beta=\tfrac12\bigl(\beta^2-1-2\log\beta\bigr)>0 .
\end{equation}
In the energy norm, with $\sigma_E$ as in \eqref{eq:enoise},
\begin{equation}\label{eq:conce}
  \Prob\bigl(\norm{e}_E>\beta\sigma_E\bigr)\le
  \exp\!\Bigl(-\tfrac12(\beta-1)^2\,\rho(S)\Bigr),
  \qquad \rho(S):=\frac{\tr(S)}{\lambda_{\max}(S)},
\end{equation}
where $\rho(S)$ is the stable rank of $S$.
\end{lemma}

\begin{proof}
The $L^2$ bound \eqref{eq:conc} is the statement that $Y:=n\norm{e}^2/\sigma^2=|\xi|^2$
is $\chi^2_n$; the tail $\Prob(Y>\beta^2n)\le e^{-c_\beta n}$ is the standard
Chernoff bound, with $c_\beta=\tfrac12(\beta^2-1-2\log\beta)>0$ for $\beta>1$
because $x-1-\log x>0$ for $x=\beta^2>1$.

For \eqref{eq:conce}, note $\norm{e}_E^2=h^2e^{\top}Se=(\sigma^2/n)\,\xi^{\top}S\xi$, so
$\norm{e}_E=(\sigma/\sqrt n)\,\abs{S^{1/2}\xi}$
with $\abs{\cdot}$ the Euclidean norm. The map $\xi\mapsto\abs{S^{1/2}\xi}$ is
Lipschitz with constant $\lambda_{\max}(S)^{1/2}$, and its mean satisfies
$\E\abs{S^{1/2}\xi}\le(\E\abs{S^{1/2}\xi}^2)^{1/2}=\tr(S)^{1/2}$. By Gaussian
concentration \cite[Thm.~5.6]{BoucheronLugosiMassart2013}, for $t>0$,
$\Prob(\abs{S^{1/2}\xi}>\tr(S)^{1/2}+t)\le\exp(-t^2/(2\lambda_{\max}(S)))$. Taking
$t=(\beta-1)\tr(S)^{1/2}$ and rescaling by $\sigma/\sqrt n$, and using
$\E\norm{e}_E^2=\sigma^2\tr(S)/n=\sigma_E^2$, gives
$\Prob(\norm{e}_E>\beta\sigma_E)\le\exp(-\tfrac12(\beta-1)^2\tr(S)/\lambda_{\max}(S))$,
which is \eqref{eq:conce}.
\end{proof}

\begin{remark}[The two concentration exponents]\label{rem:dimfree}
The $L^2$ exponent in \eqref{eq:conc} is of order $n$. The energy exponent in
\eqref{eq:conce} is of order $\rho(S)=\tr(S)/\lambda_{\max}(S)$. For the
five-point Laplacian $\tr(S)\asymp n\,h^{-2}$ and $\lambda_{\max}(S)\asymp h^{-2}$,
so $\rho(S)\asymp n$ as well, and both events are exponentially unlikely in $n$.
This is what lets the nonlinear discrepancy theorem use an energy noise level
$\beta\sigma_E$ on a high-probability event. The price is only that the level is
$\sigma_E\asymp\sigma/h$ rather than $\sigma$, which is tracked below.
\end{remark}

\paragraph{The truncated estimator} Gaussian noise is unbounded, and when
$\norm{e}$ is large the deterministic theory gives no control. We therefore
analyse a truncated estimator that uses only observable quantities. Assume a
known a~priori bound on the initial error in the sup norm,
\begin{equation}\label{eq:apriorierr}
  \norm{\xdag-x_0}_\infty\le\varrho_0,
\end{equation}
which is available from the smallness conditions and does not depend on the mesh.
Fix $k_{\max}$ and set $\varrho:=r+\varrho_0$, with $r$ the radius of the
confinement ball $B_r^\infty(\xdag)$ of \cref{prop:det}; then $x_0$ and $\xdag$
lie in $B_\varrho^\infty(x_0)$, and $B_r^\infty(\xdag)\subseteq B_\varrho^\infty(x_0)$. Run \eqref{eq:method} with the stopping rule
\eqref{eq:detstop}, the deterministic level $\delta$ replaced by the energy noise
level $\sigma_E$ of \eqref{eq:enoise}. If the rule fires at some
$k_\sigma\le k_{\max}$ and every iterate up to $k_\sigma$ lies in the observable
sup-norm ball $B_\varrho^\infty(x_0)=\{x:\norm{x-x_0}_\infty\le\varrho\}$, output
$\hat x:=\xsk[k_\sigma]$; otherwise output the fallback $\hat x:=x_0$. This test uses
only the computed iterates and the known $x_0$, so $\hat x$ is computable. Since $x_0,\xdag\in B_\varrho^\infty(x_0)$, the triangle
inequality gives $\norm{\hat x-\xdag}_\infty\le2\varrho$ always, hence
$\norm{\hat x-\xdag}\le\lvert\Om\rvert^{1/2}\,2\varrho=2\varrho=:R$ on the unit
square, a bound independent of the mesh. On the event where the deterministic
theory applies, the iterates stay in $B_r^\infty(\xdag)\subseteq B_\varrho^\infty(x_0)$,
the fallback is not used, and $\hat x$ is the discrepancy-stopped iterate.

\begin{proposition}[Discrepancy rule in expectation; full nonlinear iteration]\label{thm:discrepancy}
Let \cref{prop:det} and the order-one source condition \eqref{eq:source} hold,
together with the smallness conditions of \cite{MahaleSingh2022} for the energy
noise level $\delta_E=\beta\sigma_E$, for some $\beta\in(1,\tau)$, with
$\sigma_E=\sigma\,\overline{\lambda}(S)^{1/2}$ as in \eqref{eq:enoise}. Let
$\hat x$ be the truncated estimator above, run with the stopping threshold
$\tau\sigma_E$ in \eqref{eq:detstop}. Then there are constants $C_1>0$ and
$R=2\varrho$ with
\begin{equation}\label{eq:discrate}
  \Bigl(\E\,\norm{\hat x-\xdag}^2\Bigr)^{1/2}\;\le\;C_1\,\sigma_E^{1/2}
  \;+\;R\,e^{-\frac14(\beta-1)^2\rho(S)}.
\end{equation}
On the five-point mesh $\sigma_E\asymp\sigma/h$. For fixed $n$ the tail term does
not vanish as $\sigma\to0$, so \eqref{eq:discrate} is a joint statement in
$(\sigma,n)$; the mesh-noise coupling in the supplementary material turns it into
a rate in $\sigma$. Only the order-one source condition is covered; the
general-$\nu$ rate is the linearized statement of \cref{thm:vsc}.
\end{proposition}

\begin{proof}
The new element is a reduction to the deterministic theory on an event of high
probability, in the energy norm throughout, with the truncation controlling the
rest. The deterministic step is quoted from \cite{MahaleSingh2022}.

Let $\mathcal{E}=\{\norm{e}_E\le\beta\sigma_E\}$. On $\mathcal{E}$ the data error
satisfies the deterministic bound $\norm{y-\ys}_E\le\delta_E:=\beta\sigma_E$. The
stopping rule uses the threshold $\tau\sigma_E=(\tau/\beta)\,\delta_E$ with
$\tau/\beta>1$, so it is the deterministic discrepancy principle
\eqref{eq:detstop} in the energy norm with noise bound $\delta_E$ and constant
$\tau'=\tau/\beta>1$. The theorem \cite[Thm.~3.5]{MahaleSingh2022}, which is
stated in the Hilbert space $(\R^n,\inner{\cdot}{\cdot}_E)$, then guarantees under
the smallness conditions that the iterates stay in $B_r^\infty(\xdag)$, and hence
in $B_\varrho^\infty(x_0)$ since $\varrho\ge r+\varrho_0\ge r+\norm{\xdag-x_0}_\infty$
gives $B_r^\infty(\xdag)\subseteq B_\varrho^\infty(x_0)$, that the stopping index
is finite (take $k_{\max}$ at least the deterministic bound on it), and that
$\norm{\xsk[k_\sigma]-\xdag}_E\le(C_3\sqrt{\tau'}+C_4)\delta_E^{1/2}$. So on
$\mathcal{E}$ the fallback is not used and, by \eqref{eq:normcomp},
$\norm{\hat x-\xdag}\le c_0^{-1/2}(C_3\sqrt{\tau/\beta}+C_4)\beta^{1/2}\sigma_E^{1/2}=:C_1\sigma_E^{1/2}$.

On $\mathcal{E}^{c}$ we use only the truncation: $\norm{\hat x-\xdag}\le R$
deterministically.

For the second moment, split over $\mathcal{E}$ and $\mathcal{E}^c$:
\[
  \E\norm{\hat x-\xdag}^2
  =\E\bigl[\norm{\hat x-\xdag}^2;\mathcal{E}\bigr]
  +\E\bigl[\norm{\hat x-\xdag}^2;\mathcal{E}^{c}\bigr]
  \le C_1^2\sigma_E+R^2\,\Prob(\mathcal{E}^{c}).
\]
By \eqref{eq:conce}, $\Prob(\mathcal{E}^{c})\le e^{-\frac12(\beta-1)^2\rho(S)}$.
Taking square roots and using $(a+b)^{1/2}\le a^{1/2}+b^{1/2}$ gives
\eqref{eq:discrate}, with the tail entering as
$R\,\Prob(\mathcal{E}^c)^{1/2}=R\,e^{-\frac14(\beta-1)^2\rho(S)}$, the square root
of the probability, as a second-moment bound requires.
\end{proof}

\begin{remark}\label{rem:tail}
The tail term is the price of the truncation; it enters as the square root of a
probability because \eqref{eq:discrate} is a second-moment bound. On the grids
used below $\rho(S)\asymp n=N^2\ge31^2$, so it is far below any relevant error
level.
\end{remark}

\section{An approximate source condition, uniform in the discretization}
\label{sec:vsc}

The rates of \cref{sec:rates} rest on the range source condition
\eqref{eq:source}. On a fixed grid this condition has no content: $A_n$ is
invertible (\cref{ass:M}), so every $e_0$ lies in $\mathcal{R}(A_n)=\R^n$. The
range condition is a statement about the continuum, and we relax it there,
through a family of discretizations whose distance functions are controlled
uniformly in $n$.

\subsection{The continuum operator and a uniform family}\label{sec:family}

Let $A_\infty:=Kg'(\xdag)$ be the continuum linearized forward operator on
$L^2(\Om)$, with $K=\Lap^{-1}$. Since the continuum energy product
$\inner{u}{-\Delta u}$ is not defined on all of $L^2$, we work in $L^2(\Om)$
through the similarity
\begin{equation}\label{eq:similarcont}
  A_\infty=M^{-1/2}B_\infty M^{1/2},\qquad B_\infty:=M^{1/2}KM^{1/2},
\end{equation}
with $M=M_{g'(\xdag)}$. Under \cref{ass:M} the multiplier is bounded above and
below, so $M^{\pm1/2}$ are bounded on $L^2(\Om)$, and $B_\infty$ is compact,
positive, and self-adjoint on $L^2(\Om)$. Hence $A_\infty$ has real positive
eigenvalues $\lambda_i(A_\infty)\asymp i^{-1}$ (\cref{lem:weyl}), fractional
powers $A_\infty^\nu:=M^{-1/2}B_\infty^\nu M^{1/2}$ are well defined, and the
range $\mathcal{R}(A_\infty)=M^{-1/2}\mathcal{R}(B_\infty)$ is a proper dense
subspace of $L^2(\Om)$. The range condition $e_{0,\infty}\in\mathcal{R}(A_\infty)$
is therefore a genuine restriction: it forces $e_{0,\infty}$ to be smooth and to
respect the boundary, and a kink or a jump violates it. For such an $e_{0,\infty}$ the continuum distance function
\begin{equation}\label{eq:distcont}
  d_\infty(R):=\inf_{\norm{v}_{L^2}\le R}\norm{e_{0,\infty}-A_\infty v}_{L^2},\qquad R>0,
\end{equation}
is strictly positive for every $R$ and decays to $0$ as $R\to\infty$, at a rate
set by the residual smoothness of $e_{0,\infty}$; it is $0$ at a finite $R$
exactly when $e_{0,\infty}\in\mathcal{R}(A_\infty)$.

Let $A_n=F'(\xdag)$ be the discretization of \cref{sec:problem} on the
$N\times N$ grid, $n=N^2$, and let $e_{0,n}=\Pi_n e_{0,\infty}$ be the grid
projection of the continuum initial error. Define the discrete distance function
\begin{equation}\label{eq:distance}
  d_n(R):=\inf_{\norm{v}\le R}\,\norm{e_{0,n}-A_nv},\qquad R>0,
\end{equation}
in the discrete $L^2$ norm. On the fixed grid $d_n(R)=0$ once
$R\ge\norm{A_n^{-1}e_{0,n}}$, so $d_n$ collapses at a finite radius; that radius
grows with $n$, because $\norm{A_n^{-1}}\asymp\lambda_n^{-1}\asymp n$. What
matters is the behaviour of $d_n$ on the radii the regularization actually
visits, which we control uniformly.

\begin{assumption}[uniform approximate source condition]\label{ass:usc}
There is an index function $\psi$ (continuous, strictly decreasing on
$(0,\infty)$, $\psi(R)\to0$ as $R\to\infty$) and a constant $C_\psi$ independent
of $n$ such that
\begin{equation}\label{eq:usc}
  d_n(R)\le C_\psi\,\psi(R)\qquad\text{for all } 1\le R\le R_\sigma,
\end{equation}
for all sufficiently large $n$, where $R_\sigma$ is the balancing radius
determined by $\psi$ at noise level $\sigma$, that is, the solution of
$\alpha_\sigma R_\sigma\asymp\psi(R_\sigma)$ with $\alpha_\sigma$ the optimal
parameter of \cref{thm:vsc}. For $\psi(R)\asymp R^{-\theta}$ this is
$R_\sigma\asymp\sigma^{-1/(2\theta+1)}$, and for $\psi(R)\asymp(\log R)^{-p}$ it
is $R_\sigma\asymp\sigma^{-1}(\log\tfrac1\sigma)^{-2p}$; in both cases beyond
$R_\sigma$ the discrete collapse of $d_n$ can only help.
\end{assumption}

\Cref{ass:usc} holds whenever the continuum error obeys $d_\infty(R)\le C\psi(R)$
and the projection is stable in the sense
$\norm{(A_n\Pi_n-\Pi_nA_\infty)v}\le C h^{s}\norm{v}$ on the relevant $v$,
because then $d_n(R)\le d_\infty(R)+Ch^s R\le C\psi(R)$ for $R\le R_\sigma$ once
$h^sR_\sigma\lesssim\psi(R_\sigma)$, which couples mesh and noise. We do not prove
this projection estimate for kink and jump errors. \Cref{sec:vscnum} checks the
decay of $d_n$ numerically on one mesh; a study of its mesh independence is left
open.

\subsection{A sharp bias bound}\label{sec:sharpbias}

Fix $n$ and abbreviate $A=A_n$, $e_0=e_{0,n}$, $d=d_n$. The bias bound below is
the engine of the section; it holds on the fixed grid and is then combined with
\cref{ass:usc} to give a rate uniform in $n$.

\begin{lemma}[Sharp bias bound]\label{lem:biasvsc}
Let \cref{ass:M} hold. For the linearized iteration \eqref{eq:method}, with
the single-step residual filter $r_k(\lambda)=\alpha_{k-1}/(\lambda+\alpha_{k-1})$
of \eqref{eq:biasrec},
\begin{equation}\label{eq:biasvsc}
  b_k=\norm{r_k(A)e_0}\le\kappa_M\inf_{R>0}\bigl[\alpha_{k-1}R+d(R)\bigr]=:\kappa_M\,\Theta(\alpha_{k-1}).
\end{equation}
In the energy norm, with $d$ defined through $\norm{\cdot}_E$, the bound holds
with $\kappa_M=1$.
\end{lemma}

\begin{proof}
We use two facts about the filter $r_k(\lambda)=\alpha_{k-1}/(\lambda+\alpha_{k-1})$,
both from the spectral calculus for the symmetric matrix $B$ of
\eqref{eq:similar}, transported to $A$ by the similarity at the cost of the
factor $\kappa_M$. First, $0<r_k(\lambda)\le1$ for all $\lambda\ge0$, so
$\norm{r_k(A)}\le\kappa_M$. Second, the function
$\lambda\mapsto\lambda r_k(\lambda)=\alpha_{k-1}\lambda/(\lambda+\alpha_{k-1})$ is
increasing in $\lambda$ and tends to $\alpha_{k-1}$ as $\lambda\to\infty$, so
$\norm{Ar_k(A)}\le\kappa_M\alpha_{k-1}$.

Now fix $R>0$ and any $v$ with $\norm{v}\le R$. Split $e_0=(e_0-Av)+Av$. Since
$r_k(A)$ commutes with $A$,
\[
\begin{aligned}
  \norm{r_k(A)e_0}
  &\le\norm{r_k(A)(e_0-Av)}+\norm{r_k(A)Av}\\
  &\le\norm{r_k(A)}\,\norm{e_0-Av}+\norm{Ar_k(A)}\,\norm{v}
  \le\kappa_M\bigl(\norm{e_0-Av}+\alpha_{k-1}R\bigr) .
\end{aligned}
\]
The left-hand side does not depend on $v$. Taking the infimum over all $v$ with
$\norm{v}\le R$ turns $\norm{e_0-Av}$ into $d(R)$ and gives
$\norm{r_k(A)e_0}\le\kappa_M(\alpha_{k-1}R+d(R))$. Since $R>0$ was arbitrary, we can also
take the infimum over $R$, which gives \eqref{eq:biasvsc}. The identity
$b_k=\norm{r_k(A)e_0}$ is \eqref{eq:biasrec}.
\end{proof}

The function $\Theta$ is concave with $\Theta(0^+)=0$. On the fixed grid,
$R=\norm{A^{-1}e_0}$ gives $d(R)=0$ and hence $\Theta(\alpha)\le\alpha\norm{A^{-1}e_0}$,
but $\norm{A^{-1}e_0}$ can grow like $n$, so this bound is useless uniformly in
$n$. The useful content of \eqref{eq:biasvsc} is the value of $\Theta$ at the
balancing parameter $\alpha_\sigma$, which \cref{ass:usc} controls independently
of $n$.

The bound \eqref{eq:biasvsc} is two-sided: the bias also controls the distance
function.

\begin{lemma}[Converse]\label{lem:biasconv}
The bias controls the distance function through
\begin{equation}\label{eq:biasconv}
  d\!\left(\frac{\kappa_M\norm{e_0}}{\alpha_{k-1}}\right)\le b_k .
\end{equation}
\end{lemma}

\begin{proof}
The idea is to exhibit one specific $v_k$ for which $\norm{e_0-Av_k}$ equals the
bias, and then to read off a bound on $\norm{v_k}$. Define
\[
  v_k:=\varrho_k(A)\,e_0,\qquad
  \varrho_k(\lambda):=\frac{1-r_k(\lambda)}{\lambda}
  =\frac{\lambda/(\lambda+\alpha_{k-1})}{\lambda}=\frac{1}{\lambda+\alpha_{k-1}} .
\]
Formally $v_k=(I-r_k(A))A^{-1}e_0$, but the function $\varrho_k$ is bounded on
$[0,\infty)$, so $v_k$ is well defined for every $e_0$, including
the projection of a continuum error outside $\mathcal{R}(A_\infty)$. Then
$Av_k=\lambda\varrho_k(\lambda)$ evaluated at $A$,
which is $(1-r_k(A))e_0$, and hence
\[
  e_0-Av_k=e_0-(I-r_k(A))e_0=r_k(A)e_0,\qquad \norm{e_0-Av_k}=b_k
\]
by \eqref{eq:biasrec}. For the norm of $v_k$, note that
$0<\varrho_k(\lambda)\le\alpha_{k-1}^{-1}$ for $\lambda\ge0$, so, transporting
through \eqref{eq:similar}, $\norm{v_k}\le\kappa_M\norm{e_0}/\alpha_{k-1}=:R_k$. This means $v_k$ is one of the
competitors in the infimum that defines $d(R_k)$. Therefore
$d(R_k)\le\norm{e_0-Av_k}=b_k$, which is \eqref{eq:biasconv}.
\end{proof}

Together, \cref{lem:biasvsc,lem:biasconv} sandwich the bias between two
transforms of the distance function, up to the factor $\kappa_M$, on the radii
$R\le R_\sigma$ that the schedule visits, which is where \cref{ass:usc} applies.
With the exponent-matching duality of \cite{HofmannMathe2007,FlemmingHofmannMathe2011},
the continuum decay $\psi$ is thus both sufficient for a rate and, in the
continuum limit, necessary for it.

\subsection{Rates in expectation, uniform in the mesh}

Balancing the bias bound of \cref{lem:biasvsc} against the noise, with
\cref{ass:usc} making it uniform in $n$, gives the main result. It concerns the
linearized iteration; \cref{sec:nonlinear} examines what survives the
nonlinearity.

\begin{theorem}[Rate under an approximate source condition; linearized iteration]\label{thm:vsc}
Let \cref{ass:M,ass:usc} hold, with index function $\psi$. Set
$\Theta(\alpha)=\inf_{R>0}[\alpha R+\psi(R)]$, which is continuous, concave, and
strictly increasing with $\Theta(0^+)=0$. Then, for the linearized iteration on
the grid of size $n$, with a constant independent of $n$,
\begin{equation}\label{eq:vscrate}
  \Bigl(\E\norm{\xsk[k+1]-\xdag}^2\Bigr)^{1/2}\le\kappa_M\Bigl(C_\psi\,\Theta(\alpha_k)+\frac{\sigma}{\alpha_k}\Bigr),
\end{equation}
for $\alpha_k$ down to the optimal value $\alpha_\sigma$, whose balancing radius
is $R_\sigma$ and hence within the range controlled by \cref{ass:usc} by
construction. An a~priori choice of $k$ that balances the two terms over the
geometric schedule gives a rate determined by $\psi$. In particular:
\begin{enumerate}[label=\textup{(\roman*)},leftmargin=2.4em,itemsep=1pt]
\item \textup{(H\"older)} if $\psi(R)\asymp R^{-\nu/(1-\nu)}$ with $0<\nu<1$,
which in the continuum holds whenever $e_{0,\infty}\in\mathcal{R}(A_\infty^{\nu})$,
then $\Theta(\alpha)\asymp\alpha^{\nu}$ and the rate is $\sigma^{\nu/(\nu+1)}$,
recovering \cref{thm:apriori} and \cref{rem:nu};
\item \textup{(general power)} if $\psi(R)\asymp R^{-\theta}$, $\theta>0$, then
$\Theta(\alpha)\asymp\alpha^{\theta/(\theta+1)}$ and the rate is $\sigma^{\theta/(2\theta+1)}$;
\item \textup{(logarithmic)} if $\psi(R)\asymp(\log R)^{-p}$, $p>0$, so that
$e_{0,\infty}$ satisfies no H\"older source condition of positive order, then
$\Theta(\alpha)\asymp(\log\tfrac1\alpha)^{-p}$ and the rate is $(\log\tfrac1\sigma)^{-p}$.
\end{enumerate}
The balancing rule of \cref{app:balancing} attains the same rate without
knowledge of $\psi$ (\cref{thm:balancing}).
\end{theorem}

\begin{proof}
By \cref{lem:biasvar}(a) and Minkowski's inequality in $L^2(\Prob)$,
\[
  \bigl(\E\norm{\xsk[k+1]-\xdag}^2\bigr)^{1/2}\le b_{k+1}+\bigl(\E n_k^2\bigr)^{1/2}.
\]
The bias is deterministic and, by \cref{lem:biasvsc}, is bounded by
$\kappa_M\Theta_n(\alpha_k)$ with $\Theta_n(\alpha)=\inf_R[\alpha R+d_n(R)]$. The
unconstrained minimizer of $\alpha_k R+\psi(R)$ is $R\asymp\alpha_k^{-1/(\theta+1)}$,
which at the optimal $\alpha_k=\alpha_\sigma$ equals the radius $R_\sigma$ that
\cref{ass:usc} controls; restricting the infimum to $R\le R_\sigma$ and using
$d_n(R)\le C_\psi\psi(R)$ there gives $\Theta_n(\alpha_k)\le C_\psi\Theta(\alpha_k)$
up to a constant, for every $\alpha_k\ge\alpha_\sigma$. The noise term is bounded
by $\kappa_M\sigma/\alpha_k$ through \cref{lem:variance}. This gives
\eqref{eq:vscrate}.

The properties of $\Theta$ follow from its definition as an infimum of affine
functions of $\alpha$: it is concave and non-decreasing, $\Theta(0^+)=0$ because
$\psi(R)\to0$, and strict monotonicity follows from that of $\psi$. Concavity
gives $\Theta(\ell\alpha)\le\ell\,\Theta(\alpha)$, and the geometric schedule
\eqref{eq:alpha} always contains an index within a factor $\ell$ of any
prescribed $\alpha$. So it suffices to balance the two terms of
\eqref{eq:vscrate} as functions of a continuous parameter $\alpha$, down to
$\alpha_\sigma$, which \cref{ass:usc} covers by construction.

Case (ii). Let $\psi(R)\asymp R^{-\theta}$. The infimum
$\Theta(\alpha)=\inf_R[\alpha R+\psi(R)]$ is attained, up to constants, where the
increasing term $\alpha R$ and the decreasing term $\psi(R)$ are equal:
$\alpha R\asymp R^{-\theta}$, that is, $R\asymp\alpha^{-1/(\theta+1)}$. Then
$\Theta(\alpha)\asymp\alpha^{\theta/(\theta+1)}$. Balancing against
$\sigma/\alpha$ gives $\alpha\asymp\sigma^{(\theta+1)/(2\theta+1)}$ and the rate
$\sigma^{\theta/(2\theta+1)}$; the balancing radius is
$R_\sigma\asymp\alpha^{-1/(\theta+1)}\asymp\sigma^{-1/(2\theta+1)}$, which is
exactly the radius \cref{ass:usc} controls, for every $\theta>0$.

Case (i). Set $\theta=\nu/(1-\nu)$. Then $\Theta(\alpha)\asymp\alpha^{\nu}$ and
the rate is $\sigma^{\nu/(\nu+1)}$. That $e_{0,\infty}\in\mathcal{R}(A_\infty^{\nu})$
implies $d_\infty(R)\lesssim R^{-\nu/(1-\nu)}$ is the standard distance-function
estimate under H\"older source conditions \cite{HofmannMathe2007}; the converse
characterizes an interpolation-type class slightly larger than
$\mathcal{R}(A_\infty^{\nu})$ \cite{FlemmingHofmannMathe2011}, which we do not
need.

Case (iii). Let $\psi(R)\asymp(\log R)^{-p}$. Balancing
$\alpha R\asymp(\log R)^{-p}$ gives $R\asymp\alpha^{-1}(\log\tfrac1\alpha)^{-p}$
up to lower-order logarithmic corrections, hence
$\Theta(\alpha)\asymp(\log\tfrac1\alpha)^{-p}$; balancing against $\sigma/\alpha$
gives the rate $(\log\tfrac1\sigma)^{-p}$, with $\log(1/\alpha)\asymp\log(1/\sigma)$
at the balancing $\alpha$. The balancing radius here is
$R_\sigma\asymp\alpha_\sigma^{-1}(\log\tfrac1{\alpha_\sigma})^{-p}
\asymp\sigma^{-1}(\log\tfrac1\sigma)^{-2p}$, again the radius controlled
by \cref{ass:usc}, so the logarithmic case is covered on the same footing as the
power case.

The last claim is \cref{thm:balancing}, whose hypotheses are a non-increasing
bias bound $\kappa_M C_\psi\Theta(\alpha_{k-1})$ and a non-decreasing noise proxy.
\end{proof}

\begin{corollary}[Relaxation of the range condition; linearized iteration]\label{cor:norange}
Under \cref{ass:usc}, \cref{thm:vsc} gives a convergent rate in expectation for
the linearized iteration whenever the continuum initial error obeys the
approximate source condition $d_\infty(R)\le C\psi(R)$, and in particular for
$e_{0,\infty}\notin\mathcal{R}(A_\infty)$, at the cost of a rate slower than
$\sigma^{1/2}$ set by the decay of $\psi$. This is the precise sense in which the
range condition of \cite{MahaleSingh2022}, which excludes such $e_{0,\infty}$, is
relaxed. The rate is attained by the a~priori choice and, without knowledge of
$\psi$, by the balancing rule of \cref{thm:balancing}. We do not claim it for the
discrepancy principle (\cref{rem:discvsc}) or for the full nonlinear iteration in
this regime (\cref{sec:nonlinear}).
\end{corollary}

\begin{remark}[The discrepancy principle under the approximate source condition]\label{rem:discvsc}
We do not claim a rate for the discrepancy principle under the approximate
source condition. \Cref{thm:discrepancy} relies on the deterministic theorem of
\cite{MahaleSingh2022}, which assumes the range condition, and \cref{thm:vsc}
does not give the stopping-index inequalities the discrepancy principle needs;
such rates are delicate even deterministically
\cite{HofmannKaltenbacherResmerita2016}. The experiments of \cref{sec:vscnum}
use the discrepancy rule anyway, as numerical evidence only.
\end{remark}

\begin{remark}[Equivalence to a benchmark VSC]\label{rem:vsceq}
In the energy inner product, in which $A$ is self-adjoint, the approximate
source condition $d^E(R)\le\psi(R)$ with
$d^E(R)=\inf_{\norm{v}_E\le R}\norm{e_0-Av}_E$ is equivalent, by the duality
between distance functions and index functions
\cite{HofmannMathe2007,FlemmingHofmannMathe2011}, to a benchmark variational
source condition: there exist $\beta\in[0,1)$ and a concave index $\varphi$ with
\[
  \inner{\xdag-x_0}{\xdag-x}_{E}\le\beta\norm{x-\xdag}_{E}^2
  +\varphi\bigl(\norm{A(x-\xdag)}_E^2\bigr)\quad\forall x\in B_r^{\infty}(\xdag),
\]
$\varphi$ and $\psi$ related by the Legendre--Fenchel transform. We use the
$L^2$ form $d_n$ of \eqref{eq:distance} instead, because it is what the spectrum
gives and what enters the variance bound; since \eqref{eq:normcomp} is uniform in
one direction only, we do not claim a mesh-uniform equivalence of the two.
\end{remark}

\subsection{Which initial errors are admissible, and a numerical test}
\label{sec:vscnum}

Let $\{(\lambda_i,u_i)\}$ be the eigenpairs of $B_\infty$ in $L^2(\Om)$, with
$\lambda_i\asymp i^{-1}$ (\cref{lem:weyl}), and write
$c_i=\inner{M^{1/2}e_{0,\infty}}{u_i}$. Then
$e_{0,\infty}\in\mathcal{R}(A_\infty^{\nu})$ if and only if
$\sum_i c_i^2\lambda_i^{-2\nu}<\infty$, so the admissible H\"older order is
governed by the decay of $c_i$. A smooth error has rapidly decaying $c_i$ and
lies in $\mathcal{R}(A_\infty)$; an error with a kink, a jump, or a
non-vanishing boundary trace does not, yet still has $d_\infty(R)\to0$ and hence
a rate through \cref{thm:vsc}.

We test this on the grid ($N=49$) with $g(s)=s+s^3$, the bounded-away source
$\xdag=1+\tfrac12\sin(\pi\xi_1)\sin(\pi\xi_2)$, for which $g'(\xdag)\in[4,7.75]$ and
$\kappa_M\approx1.39$, and three continuum errors of increasing roughness: a
smooth bump, a tent (kink), and a step (jump), each scaled to $\norm{e_0}=0.05$.
For each we compute, in the eigenbasis of the symmetric matrix $B$ of
\eqref{eq:similar}, the distance function of the symmetrized problem,
$\inf_{\norm{w}\le R}\norm{M^{1/2}e_0-Bw}$, which agrees with $d_n$ up to factors
$\kappa_M$ in $R$ and in the distance and so has the same decay exponent. We fit
$R^{-\theta}$ on the large-$R$ part of the sampled curve (the upper $55\%$ of the
log-radii, where the distance still exceeds $10^{-6}$ of its maximum), predict
the rate $\sigma^{\theta/(2\theta+1)}$ of \cref{thm:vsc}(ii), and measure the rate
of the full nonlinear iteration with the discrepancy rule. The experiment goes
beyond the proven setting (the discrepancy rule and the nonlinear iteration,
\cref{rem:discvsc,sec:nonlinear}), so it tests the prediction rather than
verifying the theorem. \Cref{fig:vsc} and \cref{tab:vsc} report the outcome. The
smooth error's $d_n$ collapses early ($\theta\approx360$) and its rate is close
to $\tfrac12$. The kink and the jump, which come from continuum errors outside
$\mathcal{R}(A_\infty)$, both converge more slowly, in the order the
distance-function decay predicts. The prediction is accurate for the kink and
optimistic for the jump, whose roughness a finite grid underestimates. The
classical theory predicts no rate for either.

\begin{table}[ht]
\centering
\caption{Distance-function decay and rates for continuum initial errors of
increasing roughness, projected to an $N=49$ grid ($M=28$ realizations). The
membership column refers to the continuum operator $A_\infty$; $\theta$ is fitted
on the large-$R$ part of the computed distance curve (\cref{sec:vscnum}).}
\label{tab:vsc}
\begin{tabular}{lcccc}
\toprule
$e_{0,\infty}$ & $\in\mathcal{R}(A_\infty)$? & $\theta$ (decay of $d_n$) &
predicted $\tfrac{\theta}{2\theta+1}$ & observed $\hat p$\\
\midrule
smooth bump & yes & $361$  & $0.499$ & $0.461$\\
tent (kink) & no  & $6.66$ & $0.465$ & $0.435$\\
step (jump) & no  & $1.90$ & $0.396$ & $0.252$\\
\bottomrule
\end{tabular}
\end{table}

\begin{figure}[ht]
\centering
\includegraphics[width=\textwidth]{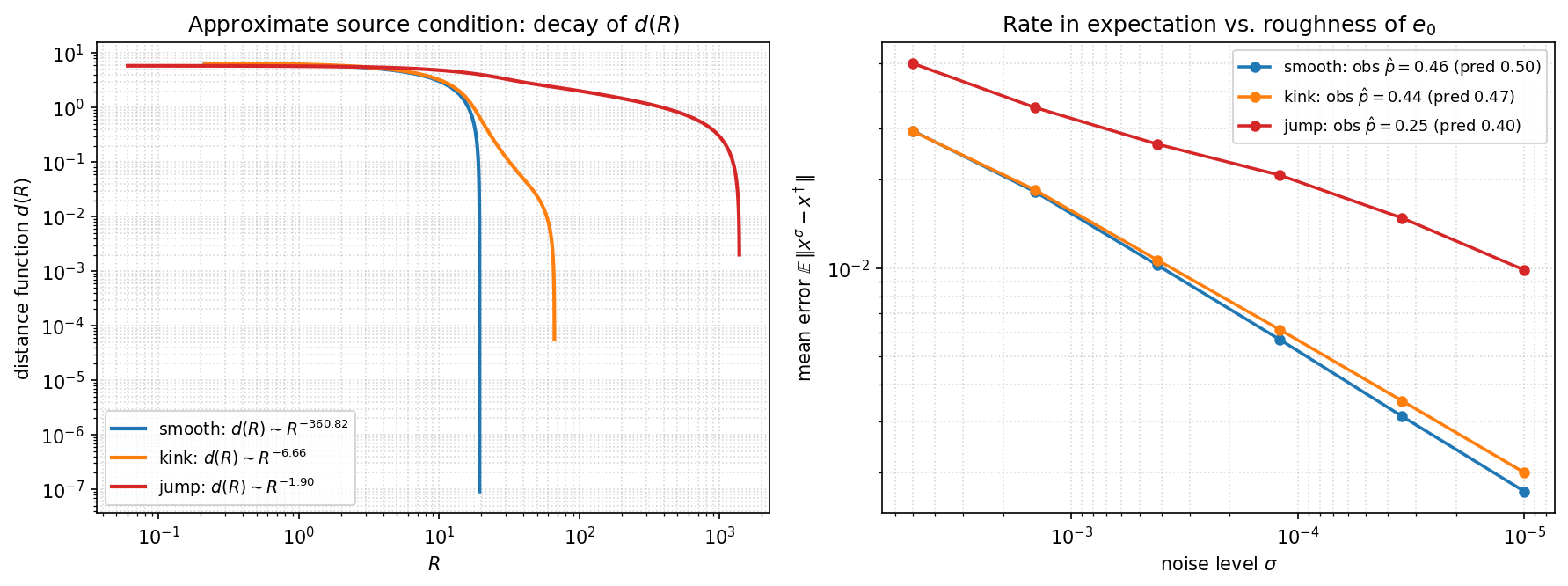}
\caption{Relaxing the range condition. Left: the computed distance function
(\cref{sec:vscnum}) decays fast for a smooth continuum error
(range element) and only polynomially for the kink and jump (continuum errors
outside $\mathcal{R}(A_\infty)$). Right: all three converge in expectation, in the
order set by the distance-function decay; for the smooth and kink errors the rate
is close to the value predicted through \cref{thm:vsc}.}
\label{fig:vsc}
\end{figure}

\subsection{The full nonlinear iteration}\label{sec:nonlinear}

The bias theory above linearizes \eqref{eq:method} about $\xdag$
(\cref{rem:linearization}). We now drop the linearization and ask how much of the
rate theory survives the iterate-dependent operator $A^{\sigma}_k=F'(\xsk)$. Write
$e_k=\xsk-\xdag$ and let $L$ be the Lipschitz constant of $F'$
(\cref{prop:det}(b)).

\paragraph{Exact error recursion} Subtract $\xdag$ from \eqref{eq:method} and
expand $F(\xsk)-F(\xdag)=A_k e_k+R_k$ with the Taylor remainder
$R_k:=F(\xsk)-F(\xdag)-F'(\xsk)(\xsk-\xdag)$. The cancellation of
\cref{lem:biasspec}, now carried out with $A_k$ in place of $A$, leaves the exact
identity
\begin{equation}\label{eq:nlrec}
  e_{k+1}=\underbrace{\alpha_k(A_k+\alpha_k)^{-1}e_0}_{\text{bias}}
  \;-\;\underbrace{(A_k+\alpha_k)^{-1}R_k}_{\text{remainder}}
  \;+\;\underbrace{(A_k+\alpha_k)^{-1}e}_{\text{noise}} .
\end{equation}
The linear part still cancels. What remains is that the bias is filtered by the
iterate-dependent $A_k$, and that the remainder $R_k$ couples the steps.

\paragraph{Term bounds} Since $A_k$ is self-adjoint and positive semidefinite
in the energy inner product, its resolvent satisfies
$\norm{(A_k+\alpha_k)^{-1}}_E\le\alpha_k^{-1}$. To
work in the $L^2$ norm we use \cref{ass:M} on the ball: for $r$ small enough
that $g'(x)\ge m_-/2$ on $B_r^{\infty}(\xdag)$, the similarity \eqref{eq:similar}
with $M_{g'(\xsk)}$ in place of $M$ gives
$\norm{(A_k+\alpha_k)^{-1}}\le\kappa'\alpha_k^{-1}$ with a constant $\kappa'$
depending only on $m_-$ and $\rho$. So the noise term obeys
$\norm{(A_k+\alpha_k)^{-1}e}\le\kappa'\norm{e}/\alpha_k$. This deterministic
operator bound sidesteps the correlation between $A_k$ and $e$. The remainder
and the operator perturbation are controlled by the $L^\infty$ Lipschitz bound
of \cref{prop:det}(b), so each carries an $L^\infty$ factor:
\[
\norm{R_k}\le\tfrac{L}{2}\norm{e_k}\,\norm{e_k}_\infty,\qquad
\norm{(A_k+\alpha_k)^{-1}R_k}\le\tfrac{\kappa'L}{2\alpha_k}\norm{e_k}\,\norm{e_k}_\infty ;
\]
and $\norm{A_k-A}\le L\norm{e_k}_\infty$, so the shifted distance function obeys
$d_k(R):=\inf_{\norm{v}\le R}\norm{e_0-A_k v}\le d(R)+L\norm{e_k}_\infty R$. This
gives the bias bound
\begin{equation}\label{eq:nlbias}
  \norm{\alpha_k(A_k+\alpha_k)^{-1}e_0}\le\kappa'\inf_{R>0}\bigl[\alpha_k R+d_k(R)\bigr]
  \le\kappa'\,\Theta\bigl(\alpha_k+L\norm{e_k}_\infty\bigr) .
\end{equation}
Collecting, and recalling $\norm{e_k}_\infty< r$ on the ball, the pathwise
recursion is
\begin{equation}\label{eq:nlmaster}
  \norm{e_{k+1}}\le\kappa'\Bigl[\Theta\bigl(\alpha_k+L\norm{e_k}_\infty\bigr)
  +\frac{\norm{e}}{\alpha_k}+\frac{L}{2\alpha_k}\norm{e_k}\,\norm{e_k}_\infty\Bigr] .
\end{equation}
The $L^\infty$ factors are the price of the nonlinearity living in the sup norm
while the rate lives in $L^2$. The next two paragraphs turn on whether they can
be made small.

\paragraph{Closure in the classical regime} When the source condition is of
order-one (range) type, that is, $\Theta(\alpha)\le c_0\alpha$, corresponding
to $\nu=1$ in \cref{thm:vsc}(i), the recursion \eqref{eq:nlmaster} closes.

\begin{proposition}[Nonlinear rate, order-one source condition]\label{prop:nonlinear}
Assume \cref{prop:det}, \cref{ass:M}, the order-one source condition
\eqref{eq:source} (so that $\Theta(\alpha)\le c_0\alpha$), and that
$L(c_0+\norm{A^{-1}e_0})$ is small enough. Let $\hat x$ be the truncated
estimator of \cref{thm:discrepancy}, run with threshold $\tau\sigma_E$. Then
\begin{equation}\label{eq:nlrate}
  \bigl(\E\norm{\hat x-\xdag}^2\bigr)^{1/2}\le C\,\sigma_E^{1/2}
  +R\,e^{-\frac14(\beta-1)^2\rho(S)},
\end{equation}
with $\sigma_E=\sigma\,\overline\lambda(S)^{1/2}\asymp\sigma/h$ and $R,\rho(S)$ as
in \cref{thm:discrepancy}. This is the rate exponent $\tfrac12$; the
general-exponent nonlinear rate is not established and is discussed in
\cref{rem:nlopen}.
\end{proposition}

\begin{proof}[(sketch)]
On the event $\mathcal{E}=\{\norm{e}_E\le\beta\sigma_E\}$ the deterministic noise
bound $\norm{y-\ys}_E\le\beta\sigma_E$ holds, and we run the deterministic
confinement of \cite[\S4]{MahaleSingh2022} in the energy norm. Let $q:=\ell^{-1}$.
Run the induction on $\norm{e_k}_\infty$: suppose $\norm{e_k}_\infty\le c_1\alpha_{k-1}$,
so also $\norm{e_k}\le\norm{e_k}_\infty\le c_1\alpha_{k-1}$ since $\lvert\Om\rvert=1$.
We bound the three terms of \eqref{eq:nlmaster}, now with $\norm{e}$ replaced by
the energy-controlled noise on $\mathcal{E}$.

Bias term. $L\norm{e_k}_\infty\le Lc_1\alpha_{k-1}\le\alpha_k$ once $Lc_1\le q$,
so $\Theta(\alpha_k+L\norm{e_k}_\infty)\le\Theta(2\alpha_k)\le2c_0\alpha_k$.

Remainder term. Using both bounds on $e_k$ and $\alpha_{k-1}\le\alpha_k/q$,
$\tfrac{L}{2\alpha_k}\norm{e_k}\norm{e_k}_\infty\le\tfrac{Lc_1^2}{2q^2}\alpha_k$.

Noise term. On $\mathcal{E}$, the propagated noise is bounded through the energy
resolvent by $\kappa'\norm{e}_E/\alpha_k\le\kappa'\beta\sigma_E/\alpha_k$. Before
the discrepancy stop $\alpha_k\ge\alpha_{k_\sigma}\asymp\sigma_E^{1/2}$, so this
is $\le c\beta\alpha_k$ for a constant $c$.

Collecting, $\norm{e_{k+1}}\le(2c_0+Lc_1^2/(2q^2)+c\beta)\alpha_k\le c_1\alpha_{k+1}$
once $c_1$ is large compared with $c_0/q$, $L$ is small, and $\beta$ is small
compared with $c_1q$. Confinement in the sup norm is the machinery of
\cite[\S4]{MahaleSingh2022}, which in this regime keeps $\norm{e_k}_\infty$ at the
scale of $\norm{e_k}$. At the stop, bias and noise balance at $\sigma_E^{1/2}$, so
$\norm{\hat x-\xdag}\le C\sigma_E^{1/2}$ on $\mathcal{E}$.

On $\mathcal{E}^c$ we do not use confinement, which is exactly what fails there;
we use the truncation of \cref{thm:discrepancy}, giving $\norm{\hat x-\xdag}\le R$.
Splitting the second moment as in that proof,
$\E\norm{\hat x-\xdag}^2\le C^2\sigma_E+R^2\Prob(\mathcal{E}^c)$, and
$\Prob(\mathcal{E}^c)\le e^{-\frac12(\beta-1)^2\rho(S)}$ by \eqref{eq:conce}.
Taking square roots gives \eqref{eq:nlrate}, the tail entering as
$R\,\Prob(\mathcal{E}^c)^{1/2}=R\,e^{-\frac14(\beta-1)^2\rho(S)}$.
\end{proof}

\paragraph{The rough regime: a real obstruction} For a proper approximate source
condition, $\Theta(\alpha)/\alpha\to\infty$ as $\alpha\to0$. This is the regime
\cref{thm:vsc} is built for, and here the closure fails. The bias inflation
\eqref{eq:nlbias} is harmless only while $L\norm{e_k}_\infty\lesssim\alpha_k$,
whereas the rate forces $\norm{e_k}\sim\Theta(\alpha_{k-1})\gg\alpha_{k-1}$, so
the two are incompatible already in $L^2$. The $L^\infty/L^2$ asymmetry makes it
worse, since the inflation is driven by $\norm{e_k}_\infty$, which in two
dimensions is not controlled by the $L^2$ rate. The deterministic confinement of
the classical regime is therefore not available, the operator-linearization error
$\norm{A_k-A}=O(L\norm{e_k}_\infty)$ is no longer of higher order, and
\eqref{eq:nlmaster} does not close. For rough initial errors, \cref{thm:vsc}
remains a statement about the linearization.

\begin{remark}[Status of the rough nonlinear case]\label{rem:nlopen}
This is an obstruction of the technique, not evidence that the method fails:
the experiments of \cref{sec:vscnum} use the full nonlinear iteration and
converge for the kink and jump errors (\cref{tab:vsc}). A proof would presumably
need a range-type nonlinearity condition that controls $F'(\cdot)-F'(\xdag)$ in
the data norm rather than through $\norm{e_k}$; this is open.
\end{remark}

\section{Optimality and numerical validation}\label{sec:validation}

\subsection{Optimality and relation to known rates}\label{sec:optimality}

Whether the rate $\sigma^{\nu/(\nu+1)}$ is optimal depends on the noise model.
We separate three settings. (i)
Deterministic worst-case noise of norm $\le\delta$. (ii) The Gaussian
finite-energy model $\E\norm{e}^2=\sigma^2$ used in this paper. (iii) The Gaussian
white-noise model with variance $\sigma^2$ per mode. The method is order-optimal in
(i), away from the parametric floor. In (ii) it is the natural estimator but is
statistically inefficient relative to spectral cut-off, as we show by comparison
rather than by a minimax theorem. In (iii) it does not attain the minimax rate.

\paragraph{(i) Deterministic worst-case noise} Since $\norm{e}$ concentrates
about $\sigma$ (\cref{lem:conc}), a typical realization is a perturbation of
norm close to $\sigma$, and the natural benchmark is the classical worst-case
lower bound under a source condition \cite{EnglHankeNeubauer1996}, which we prove
for the present operator.

\begin{proposition}[Worst-case lower bound]\label{prop:lower}
Let \cref{ass:M} hold. Linearize about $\xdag$, and let
$\mathcal{M}_{\nu,R}=\{x:x-x_0=A^{\nu}v,\ \norm{v}\le R\}$. Let $c_\pm$ be the
constants of \cref{lem:weyl}, set $c_\ast:=2c_+/c_-$ and
$\lambda^\ast=(\delta/R)^{1/(\nu+1)}$, and suppose $c_\ast\lambda_n\le\lambda^\ast<c_+$,
which holds for $R\,(c_\ast\lambda_n)^{\nu+1}\le\delta<R\,c_+^{\nu+1}$, a range whose
lower end is $\asymp R\,n^{-(\nu+1)}$.
Then for data perturbations of $L^2$ norm $\le\delta$, every reconstruction
$\mathcal{R}$ obeys
\begin{equation}\label{eq:lower}
  \sup_{x\in\mathcal{M}_{\nu,R},\ \norm{e}\le\delta}\norm{\mathcal{R}(F(x)+e)-x}
  \;\ge\;c_\ast^{-\nu}\,\delta^{\nu/(\nu+1)}R^{1/(\nu+1)}.
\end{equation}
\end{proposition}

\begin{proof}
The argument is the classical two-point construction. We build two admissible
sources that are far apart but produce the same data after an admissible
perturbation, so no method can tell them apart.

We first show that $[\lambda^\ast/c_\ast,\lambda^\ast]$ contains an eigenvalue.
The bounds $c_-/i\le\lambda_i\le c_+/i$ of \cref{lem:weyl} are not enough to keep
consecutive ratios below $2$ (on the grids used here $\lambda_1/\lambda_2\approx2.8$),
so we count. Every index $i\le c_-c_\ast/\lambda^\ast$ has
$\lambda_i\ge\lambda^\ast/c_\ast$, and these indices are admissible because
$c_-c_\ast/\lambda^\ast\le c_-/\lambda_n\le n$; so at least
$\lfloor 2c_+/\lambda^\ast\rfloor\ge 2c_+/\lambda^\ast-1$ eigenvalues are
$\ge\lambda^\ast/c_\ast$. On the other hand $\lambda_i>\lambda^\ast$ forces
$i<c_+/\lambda^\ast$, so fewer than $c_+/\lambda^\ast$ eigenvalues exceed
$\lambda^\ast$. Since $c_+/\lambda^\ast>1$, the difference is positive, and some
eigenvalue $\lambda$ lies in $[\lambda^\ast/c_\ast,\lambda^\ast]$. The
construction fails once $\delta$ is so small that $\lambda^\ast<c_\ast\lambda_n$. Pick the corresponding normalized eigenmode $u$, $Au=\lambda u$,
$\norm{u}=1$, and set
\[
  x_\pm:=x_0\pm\lambda^{\nu}R\,u .
\]
Since $A^{\nu}u=\lambda^{\nu}u$, we have $x_\pm-x_0=A^{\nu}(\pm Ru)$ with
$\norm{\pm Ru}=R$, so both $x_+$ and $x_-$ lie in $\mathcal{M}_{\nu,R}$.

Now compute the data. In the linearization about $\xdag$, $F(x)=F(x_0)+A(x-x_0)$,
so $F(x_\pm)=F(x_0)\pm\lambda^{1+\nu}R\,u$. Choose $e_\pm:=\mp\lambda^{1+\nu}R\,u$,
with $L^2$ norm $\norm{e_\pm}=\lambda^{1+\nu}R\le(\lambda^\ast)^{1+\nu}R=\delta$
since $\norm{u}=1$, admissible for the $L^2$ constraint. Then
$F(x_+)+e_+=F(x_0)=F(x_-)+e_-$: the two experiments produce identical data.

Any reconstruction $\mathcal{R}$ maps these identical data to a single element
$z:=\mathcal{R}(F(x_0))$. By the triangle inequality,
$\norm{x_+-x_-}\le\norm{z-x_+}+\norm{z-x_-}$, so at least one of the two errors
is at least $\tfrac12\norm{x_+-x_-}$. Therefore the worst-case error of
$\mathcal{R}$ over $\mathcal{M}_{\nu,R}$ is at least
\[
  \tfrac12\norm{x_+-x_-}=\lambda^{\nu}R\ge c_\ast^{-\nu}(\lambda^\ast)^{\nu}R
  =c_\ast^{-\nu}\Bigl(\frac{\delta}{R}\Bigr)^{\nu/(\nu+1)}R
  =c_\ast^{-\nu}\,\delta^{\nu/(\nu+1)}R^{1/(\nu+1)},
\]
using $\lambda\ge\lambda^\ast/c_\ast$ in the inequality and the definition of
$\lambda^\ast$ in the last two equalities. This is \eqref{eq:lower}.
\end{proof}

This $L^2$ bound matches the linearized rate of \cref{thm:apriori} with
$\delta=\sigma$, both scaling as $\sigma^{\nu/(\nu+1)}$, as long as
$\sigma\gtrsim R\lambda_n^{\nu+1}$. Below that level the construction fails
because $\lambda^\ast<c_\ast\lambda_n$, and on a fixed grid an optimal estimator
eventually reaches a parametric rate of order $\sigma$. So the linearized
a~priori estimator, and the approximate-source rates of \cref{sec:vsc} for their
distance functions, are order-optimal against worst-case $L^2$ noise. The
nonlinear rate of \cref{thm:discrepancy} is stated in $\sigma_E$ and would need an
energy-norm version of the construction, which we do not carry out. In either
norm, \cref{prop:lower} concerns adversarial noise and says nothing about minimax
optimality for Gaussian noise.

\paragraph{(ii) and (iii) Gaussian noise: statistical inefficiency} The resolvent
$(A+\alpha)^{-1}$ damps every mode instead of discarding the unresolved ones, so
the propagated-noise variance is tied to $\sigma^2/\alpha^2$
(\cref{app:sharpvar}), whereas a spectral cut-off estimator has variance set by
the effective dimension \eqref{eq:neff}. We do not prove a minimax theorem for
the finite-energy model; we compare with cut-off instead, in the eigenbasis of
the symmetrized operator $B$ of \eqref{eq:similar} for the problem of
\cref{sec:vscnum}, with isotropic noise, $\sigma=10^{-3}$, and one fixed truth
$x-x_0=Av$ discretized on every mesh. Refining the mesh from $n=225$ to $n=2809$
leaves the oracle Lavrentiev error essentially unchanged, while the oracle
cut-off error falls, so their ratio grows: from $21$ to $51$ for a smooth $v$,
from $13$ to $36$ for a $v$ with a kink, and from $4.5$ to $8.7$ for a $v$ with a
jump (scripts \texttt{experiment7\_fixed.py} and \texttt{experiment7\_rough.py}).
The size of the gap depends on the truth; its growth with $n$ does not. So the
method is statistically inefficient relative to cut-off in the regimes examined,
not proven minimax-suboptimal. These rates are meaningful while $\alpha_{k_\sigma}\gg\lambda_n$,
that is $\sigma^{1/2}\gg n^{-1}$ (\cref{rem:operating}), or under a mesh-noise
coupling. In the white-noise model, with variance $\sigma^2$ per mode, singular
values $\lambda_i\asymp i^{-\beta}$, and smoothness $\nu$, the minimax $L^2$ rate
is $\sigma^{2\beta\nu/(2\beta\nu+2\beta+1)}$
\cite{Cavalier2011,CavalierGolubevPicardTsybakov2002,Tsybakov2009}, that is
$\sigma^{2\nu/(2\nu+3)}$ for $\beta=1$; since the total noise energy there grows
with $n$, this exponent is not directly comparable with ours. The case for
iterated Lavrentiev is cost, not statistical efficiency: it is adjoint-free and
forms no normal equations, and it needs an order of magnitude fewer solver
applications in the comparison of the supplementary material. A
variance-aware parameter choice for an adjoint-free iteration, with a matching
Gaussian lower bound, is an open problem.

\subsection{Numerical experiments}\label{sec:numerics}

\paragraph{Set-up} We discretize $\Om=(0,1)^2$ by the five-point Laplacian on an
$N\times N$ interior grid ($n=N^2$). This gives the SPD matrix $S\approx\Lap$ and
the forward map $F(x)=S^{-1}x^3$. Each iteration of \eqref{eq:method} reduces to
one SPD solve $(\,\mathrm{diag}(g'(\xsk))+\alpha_k S)\,z=S r$, with no adjoint.
We take $N=49$, the geometric schedule $\alpha_k=(0.7)^k$ ($\ell=10/7$),
$\tau=1.4$, the exact source $\xdag(\xi)=\sin(\pi\xi_1)\sin(\pi\xi_2)$, and the
source-condition-consistent guess $x_0=\xdag+K(g'(\xdag))$, that is, $v\equiv1$
in \eqref{eq:source}. The noise is generated as in \eqref{eq:noisegen},
$e=\sigma\,\xi/(h\sqrt n)$ with $\xi\sim\mathcal N(0,I_n)$ drawn afresh for each
run, so that each realization is Gaussian with $\E\norm{e}^2=\sigma^2$.
Realizations are never rescaled, so the rules receive a fixed level and not the
realized norm. The level $\sigma$ takes six values from $10^{-2}$ to $10^{-5}$,
with $M=80$ realizations each. \Cref{fig:recon} shows a typical reconstruction.

\begin{figure}[ht]
\centering
\includegraphics[width=0.93\textwidth]{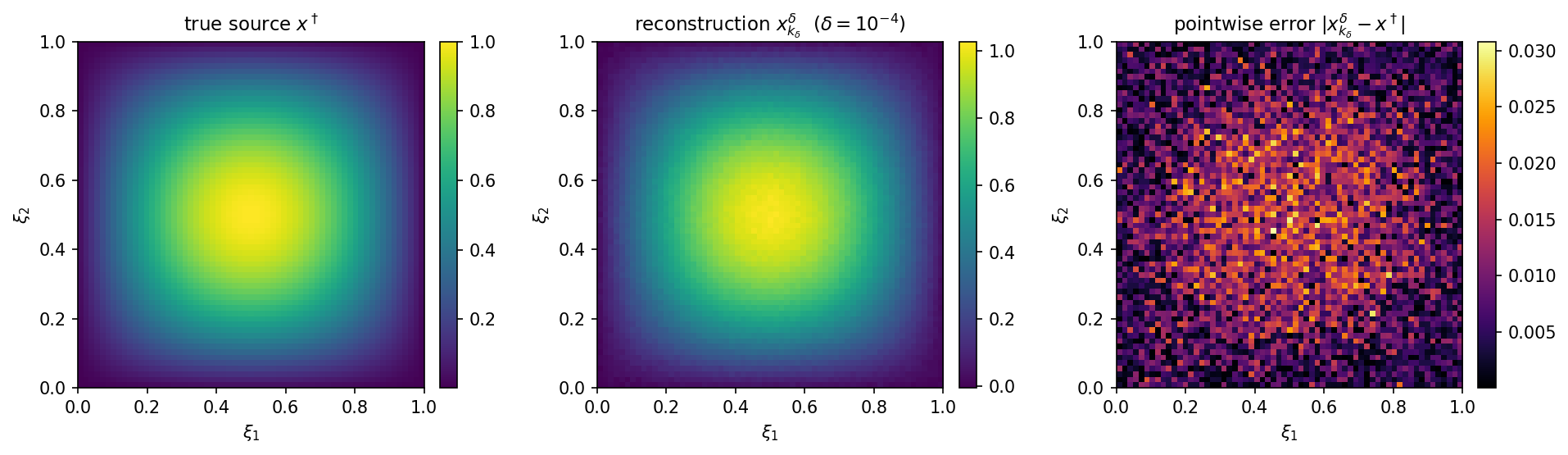}
\caption{The inverse source problem: true source $\xdag$, a reconstruction, and
the pointwise error at a representative noise level. This is the recovery that
the statistical analysis quantifies.}
\label{fig:recon}
\end{figure}

\paragraph{Stopping rules compared} For each realization we compute four
indices: the (infeasible) oracle $k=\arg\min_k\norm{\xsk-\xdag}$; the
discrepancy rule \eqref{eq:detstop} with $\delta=\sigma$; the balancing rule
\eqref{eq:balancing} with $\Phi(k)=\kappa\sigma\phi_k$ ($\kappa=0.3$, with
$\phi_k$ precomputed by probing the reference resolvent); and the parameter-free
quasi-optimality heuristic $k=\arg\min_k\norm{\xsk[k+1]-\xsk}$. Two of these are
practical variants rather than the rules analysed. The discrepancy rule is fed
$\delta=\sigma$, whereas \cref{thm:discrepancy} uses $\sigma_E$, which is
$100\,\sigma$ for $N=49$; and the balancing rule uses the calibrated value
$\kappa=0.3$, below the admissible $\kappa\ge2\kappa_M$ of \cref{thm:balancing}
(\cref{rem:kappa}). The experiment therefore illustrates these results rather
than verifying them; in addition, the source vanishes on the boundary, so
\cref{ass:M} holds only with a mesh-dependent constant.

\paragraph{Results} \Cref{tab:stat} and \cref{fig:rate} report the Monte Carlo
mean error of each rule and the exponent $\hat p$ fitted in
$\E\norm{\cdot}\sim\sigma^{\hat p}$. The discrepancy rule tracks the oracle
almost exactly ($\hat p=0.397$ against $0.384$; the means coincide to plotting
accuracy for $\sigma\le1.6\times10^{-4}$), and the truncation was never triggered
in $480$ runs. A run fed $\sigma_E$ itself stops much earlier and probes a
different part of the schedule; it is left to future work. The balancing rule
tracks the oracle within a constant ($\hat p=0.336$); its index increases
monotonically ($\bar k_{\mathrm{bal}}=0,1,6,9,12,14$ as $\sigma$ decreases),
following the bias-variance trade-off of \eqref{eq:oracle} shown in
\cref{fig:tradeoff}. Quasi-optimality stalls at an early iterate (error about $0.050$, against
$\norm{e_0}\approx0.055$) for most noise levels and drops only at the smallest
$\sigma$, a warning against parameter-free rules in this regime.

\begin{table}[ht]
\centering
\caption{Monte Carlo estimate of $\E\,\norm{\xsk[k_\sigma]-\xdag}$ over $M=80$
realizations ($N=49$). Bottom row: fitted exponent $\hat p$ in
$\E\norm{\cdot}\sim\sigma^{\hat p}$ with standard error. The discrepancy rule
matches the oracle; the balancing rule attains the rate within a constant;
quasi-optimality fails.}
\label{tab:stat}
\small
\begin{tabular}{lcccc}
\toprule
$\sigma$ & oracle & discrepancy ($\hat\delta{=}\sigma$) & balancing (Lepskii)
& quasi-optimality \\
\midrule
$1.00\times10^{-2}$ & $5.02\times10^{-2}$ & $5.49\times10^{-2}$ & $5.49\times10^{-2}$ & $5.10\times10^{-2}$\\
$2.51\times10^{-3}$ & $3.74\times10^{-2}$ & $4.07\times10^{-2}$ & $5.01\times10^{-2}$ & $5.01\times10^{-2}$\\
$6.31\times10^{-4}$ & $2.31\times10^{-2}$ & $2.41\times10^{-2}$ & $3.54\times10^{-2}$ & $5.00\times10^{-2}$\\
$1.59\times10^{-4}$ & $1.31\times10^{-2}$ & $1.31\times10^{-2}$ & $2.17\times10^{-2}$ & $5.00\times10^{-2}$\\
$3.98\times10^{-5}$ & $7.14\times10^{-3}$ & $7.14\times10^{-3}$ & $1.06\times10^{-2}$ & $5.00\times10^{-2}$\\
$1.00\times10^{-5}$ & $3.71\times10^{-3}$ & $3.79\times10^{-3}$ & $5.98\times10^{-3}$ & $3.79\times10^{-3}$\\
\midrule
$\hat p$ & $0.384{\scriptstyle\pm0.022}$ & $0.397{\scriptstyle\pm0.021}$
& $0.336{\scriptstyle\pm0.039}$ & $0.269{\scriptstyle\pm0.153}$\\
\bottomrule
\end{tabular}
\end{table}

\begin{figure}[ht]
\centering
\begin{subfigure}{0.49\textwidth}\centering
\includegraphics[width=\textwidth]{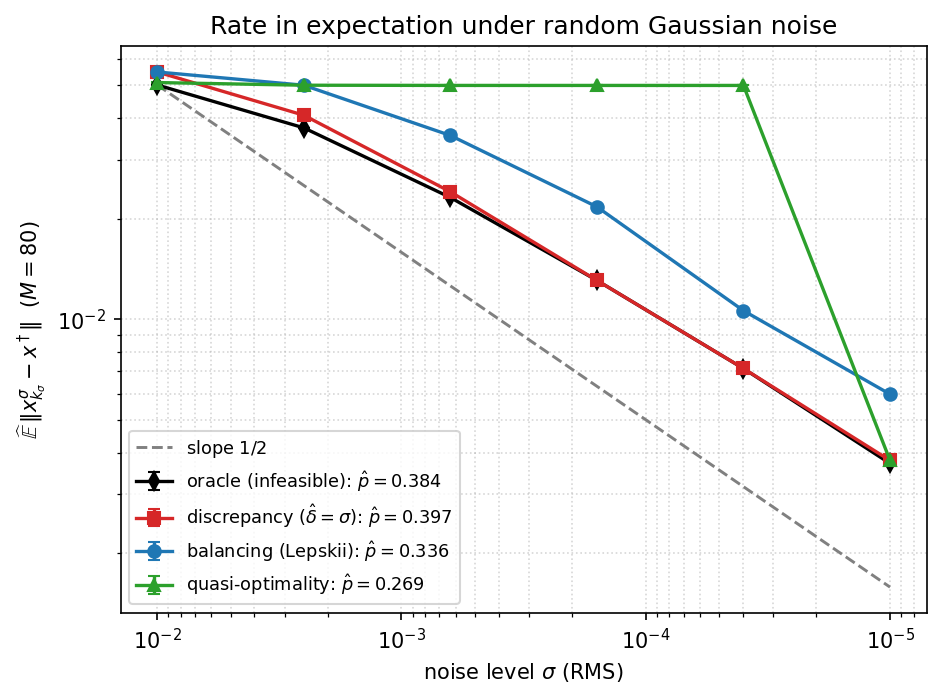}
\caption{Rate in expectation.}\label{fig:rate}
\end{subfigure}\hfill
\begin{subfigure}{0.49\textwidth}\centering
\includegraphics[width=\textwidth]{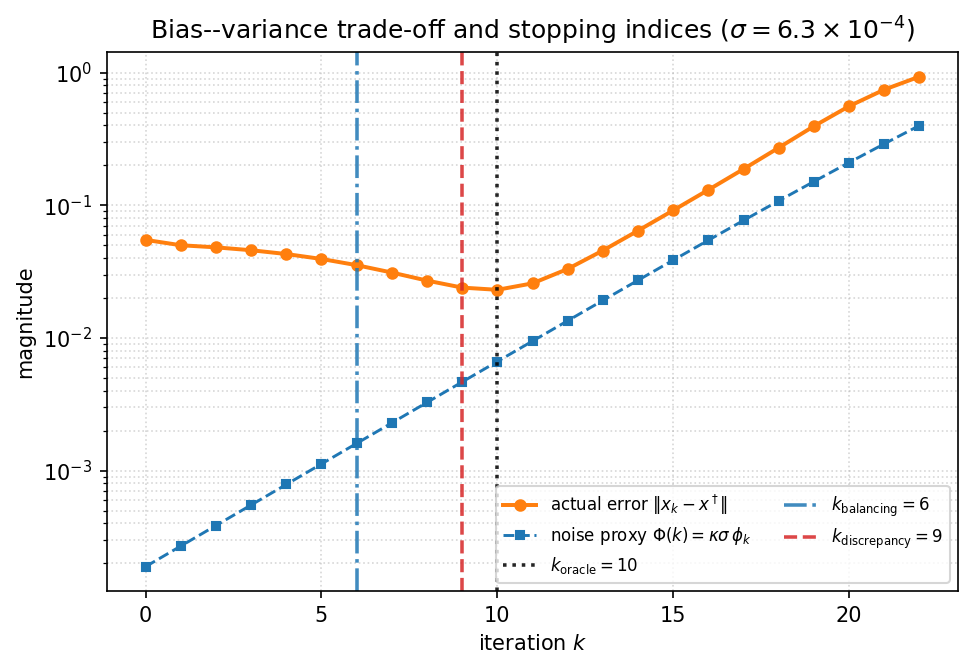}
\caption{Bias-variance trade-off, one realization.}\label{fig:tradeoff}
\end{subfigure}
\caption{Left: Monte Carlo mean error against noise level. The discrepancy rule
(red) lies on top of the oracle (black), the balancing rule (blue) is parallel
within a constant, and quasi-optimality (green) is erratic. Right: the actual
error (orange) is U-shaped in $k$ as bias decreases and propagated noise, tracked
by the proxy $\Phi(k)$ (blue), grows. The oracle, balancing, and discrepancy
indices all fall near the minimum.}
\label{fig:rate-trade}
\end{figure}

\paragraph{Concentration and reliability} \Cref{fig:box} shows the error
distributions over the $M=80$ realizations. They are tightly concentrated, with
standard deviations of at most $0.9\%$ of the mean for the oracle and
discrepancy rules (zero at $\sigma=10^{-2}$, where the discrepancy rule stops at
$k=0$ in every run and returns $x_0$), as \cref{thm:conc} predicts for the linearized estimator
(\cref{rem:concnum}).

\begin{figure}[ht]
\centering
\includegraphics[width=0.96\textwidth]{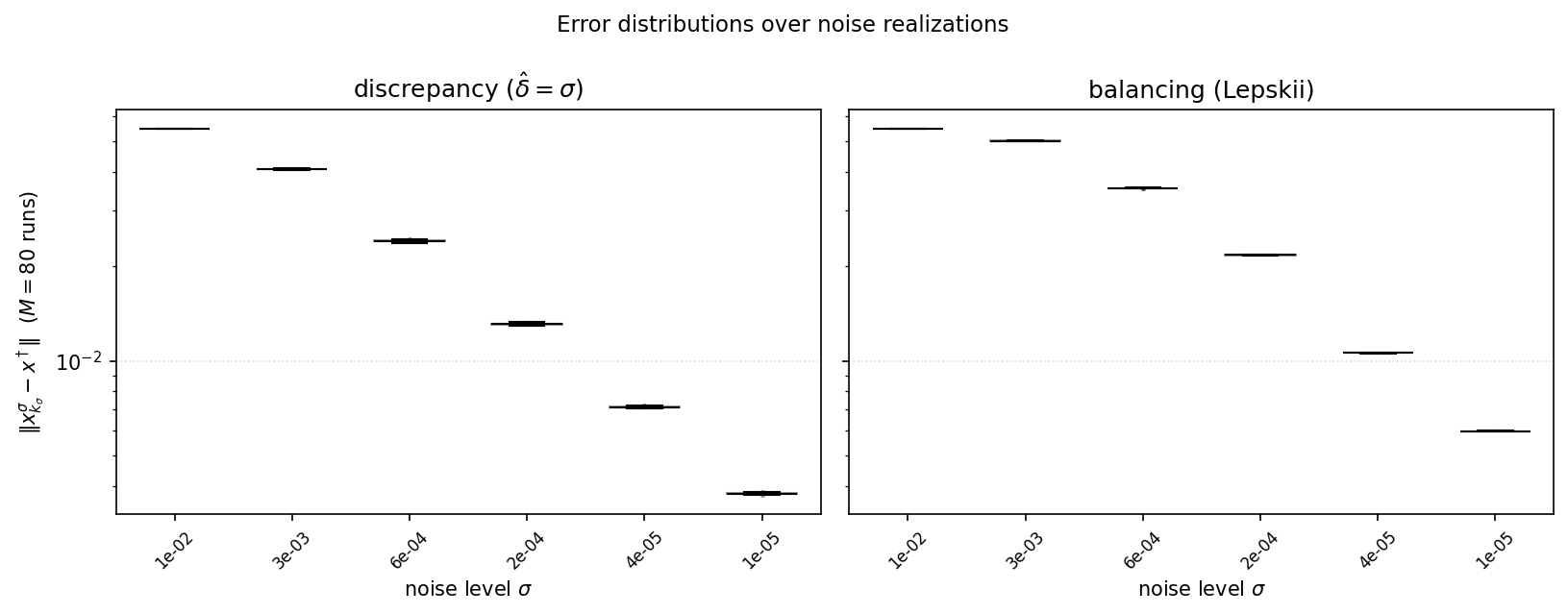}
\caption{Error distributions over $M=80$ noise realizations for the discrepancy
(left) and balancing (right) rules. The tight spread shows that the
expected-error rates are representative of individual realizations.}
\label{fig:box}
\end{figure}

\paragraph{On the observed exponent} The fitted exponents ($0.38$ to $0.40$ for
the oracle and discrepancy rules) lie below the proven upper-bound exponent
$\tfrac12$. This neither confirms nor contradicts the theory. The fit covers three
decades of $\sigma$ on a parameter grid with ratio $0.7$, it is stable across
rules and, in companion deterministic runs, across mesh refinements, and a
local-slope diagnostic on \cref{fig:rate} shows the slope increasing toward the
smallest noise levels without reaching $\tfrac12$. Smaller $\sigma$ would require
the mesh-noise coupling of the supplementary material, since on a fixed mesh the
discretization floor intervenes first. We see no systematically faster rate,
consistent with the inefficiency relative to cut-off.

\paragraph{A second test problem, with uncertainty quantification} To check
that the rate is not specific to the cubic nonlinearity or to the symmetric
source of the main set-up, we repeat the experiment with a different monotone
nonlinearity, $g(x)=\sinh x$, an asymmetric multi-bump source $\xdag$, and the
uninformed guess $x_0\equiv1$, on the same operator $-\Delta$ ($N=40$, $M=40$
Gaussian realizations per level, $\alpha_0=0.5$, ratio $0.7$). In this experiment
the noise level and the errors are measured in the Euclidean norm of the grid
vector, which is $1/h=41$ times the discrete $L^2$ norm; this rescales all errors
but leaves exponents and relative spreads unchanged. \Cref{tab:test2} reports the Monte
Carlo mean and standard deviation of the oracle error. The fitted exponent is
$0.52$, in agreement with the predicted $\tfrac12$. The standard deviations are
between $0.7\%$ and $1.3\%$ of the mean, and the $95\%$ confidence half-widths,
about $1.96\,\mathrm{s.d.}/\sqrt{M}$, are a third of that. Here the discrepancy
rule is the plain residual test $\norm{F(\xsk)-\ys}\le1.3\sigma$. It is fed only
$\sigma$ and is markedly less accurate than the oracle, by a factor between $2.3$
and $4.7$ over this range, although its error decreases faster (fitted exponent
$0.68$).

\begin{table}[ht]
\centering
\caption{Second test problem ($g=\sinh$, asymmetric source, $N=40$, $M=40$):
Monte Carlo mean $\pm$ standard deviation of the oracle error, with noise levels
and errors in the Euclidean grid norm. Fitted rate
exponent $0.52$ (predicted $0.50$).}
\label{tab:test2}
\begin{tabular}{lccccc}
\toprule
$\sigma$ & $10^{-2}$ & $3\!\times\!10^{-3}$ & $10^{-3}$ & $3\!\times\!10^{-4}$ & $10^{-4}$\\
\midrule
oracle error & $1.886$ & $1.058$ & $0.611$ & $0.319$ & $0.169$\\
$\pm$ s.d.\ & $0.013$ & $0.009$ & $0.007$ & $0.004$ & $0.002$\\
\bottomrule
\end{tabular}
\end{table}

\section{Conclusion}\label{sec:conclusion}

Under random noise, the convergence of the linearized iterated Lavrentiev method
is governed by the decay of a distance function rather than by a binary range
condition, provided this decay is uniform across discretizations
(\cref{ass:usc}). The rate in expectation then degrades gracefully, down to
logarithmic, as the continuum error loses smoothness, and the classical H\"older
theory is the case in which the distance function collapses (\cref{thm:vsc}).
For the full nonlinear iteration under the order-one source condition,
energy-norm concentration and a computable truncation replace the missing
pathwise bound, and the discrepancy principle fed only a noise level attains
$\sigma_E^{1/2}$ jointly in $(\sigma,n)$. The main open questions, each stated
where it arises, are a proof of \cref{ass:usc} from the continuum operator, the
nonlinear rate for rough errors and for $\nu<1$ (\cref{rem:nlopen}), the
discrepancy principle under the approximate source condition
(\cref{rem:discvsc}), and a Gaussian minimax analysis (\cref{sec:optimality}).

\appendix
\makeatletter
\renewcommand{\thetheorem}{\Alph{section}.\arabic{theorem}}
\makeatother
These appendices collect supporting material, each tied to a claim in the main
text. \Cref{app:sharpvar} sharpens the variance bound of \cref{lem:variance}
through the spectrum. \Cref{app:balancing} gives an adaptive (Lepskii)
alternative to the discrepancy rule of \cref{thm:discrepancy}, and
\cref{app:conc} gives a concentration bound showing that the linearized estimator
of \cref{sec:rates} is sharply peaked about its mean. Further numerical
validation of the rate theory, its extension to heterogeneous media with a
total-error and mesh-coupling theorem, and a large-scale comparison against an
adjoint-based scheme are given in the supplementary material; these support, but
are not required by, the main result of \cref{sec:vsc}.

\section{Sharp variance and the operating regime}\label{app:sharpvar}

\begin{proposition}[Sharp variance]\label{prop:sharpvar}
The variance proxy $\phi_k^2=\tfrac1n\sum_i(\lambda_i+\alpha_k)^{-2}$ of
\cref{lem:variance} obeys:
\begin{enumerate}[label=\textup{(\roman*)},leftmargin=2.4em,itemsep=1pt]
\item if $\alpha_k\gtrsim\lambda_1$, then $\phi_k^2\asymp\alpha_k^{-2}$ and the crude
bound is sharp;
\item if $\alpha_k\lesssim\lambda_n$, then $\phi_k^2\asymp\frac1n\sum_i\lambda_i^{-2}$
is constant. The variance saturates, and the crude bound $\alpha_k^{-2}$
overestimates it by the unbounded factor $\asymp n\,\alpha_k^{-2}\bigl(\sum_i\lambda_i^{-2}\bigr)^{-1}$.
\end{enumerate}
\end{proposition}

\begin{proof}
In case (i) every eigenvalue satisfies $\lambda_i\le\lambda_1\le c\,\alpha_k$ for
some constant $c$, so $\alpha_k\le\lambda_i+\alpha_k\le(1+c)\alpha_k$ and each
summand $(\lambda_i+\alpha_k)^{-2}$ lies between $(1+c)^{-2}\alpha_k^{-2}$ and
$\alpha_k^{-2}$. There are $n$ summands and we divide by $n$, so
$\phi_k^2\asymp\alpha_k^{-2}$, and the crude upper bound is attained up to the
constant $(1+c)^{-2}$.

In case (ii) every eigenvalue satisfies $\lambda_i\ge\lambda_n\ge c'\alpha_k$, so
$\lambda_i\le\lambda_i+\alpha_k\le(1+1/c')\lambda_i$ and each summand is
$\asymp\lambda_i^{-2}$, with constants that do not depend on $\alpha_k$. Hence
$\phi_k^2\asymp\frac1n\sum_i\lambda_i^{-2}$, a quantity fixed by the spectrum
alone. As $\alpha_k\downarrow0$ this stays bounded while the crude bound
$\alpha_k^{-2}$ grows without limit. The ratio of the two is
$\alpha_k^{-2}/\bigl(\frac1n\sum_i\lambda_i^{-2}\bigr)=n\,\alpha_k^{-2}
\bigl(\sum_i\lambda_i^{-2}\bigr)^{-1}$, up to the stated constants.
\end{proof}

\begin{remark}[Operating point]\label{rem:operating}
At the discrepancy or balancing level $\alpha_{k_\sigma}\asymp\sigma^{1/2}$ one
has, for the noise levels of interest, $\lambda_n\ll\alpha_{k_\sigma}$. So the
dominant block of high-index modes ($\lambda_i\ll\alpha_{k_\sigma}$) makes
$\phi_{k_\sigma}^2\asymp\alpha_{k_\sigma}^{-2}$. The crude variance bound is
therefore sharp at the operating point, and the rate $\sigma^{1/2}$ cannot be
improved by the refinement. The saturation (ii) governs only the
over-regularized tail $k\gg k_\sigma$. Both predictions are borne out in
the supplementary material, where the Monte Carlo variance matches $\sigma^2\phi_k^2$ to
within $0.7\%$ and departs from the crude bound by a factor $7.3$ at the smallest
$\alpha_k$.
\end{remark}

\section{A balancing rule adaptive to smoothness}\label{app:balancing}

When neither the smoothness nor a sharp variance proxy is known, the Lepskii
balancing principle \cite{Lepskii1990,MathePereverzev2003} selects the index
adaptively. We state it for the linearized iteration, where the error
decomposition is exact. Let $\beta_k$ be a known non-increasing bound on the
bias, $b_k\le\beta_k$; under the classical source condition
$\beta_k=\kappa_M\alpha_{k-1}\norm{v}$ (\cref{lem:biasspec}), and under the
approximate source condition $\beta_k=\kappa_M\Theta(\alpha_{k-1})$
(\cref{lem:biasvsc}). Let $\Phi(k)$ be a non-decreasing bound on the propagated
noise. By \cref{lem:variance} we may take $\Phi(k)=\kappa\,\sigma\,\phi_k$ for a
constant $\kappa\ge2\kappa_M$, which is what \cref{thm:balancing} requires; note
that $\phi_{k-1}\le\phi_k$, so this also bounds the root mean square of $n_{k-1}$,
the noise entering $\xsk$ by \eqref{eq:linsplit}. The proxy $\phi_k$ is computed
from the reference operator $A=F'(\xdag)$, which depends on the unknown $\xdag$;
so the rule as analysed is adaptive to the smoothness but not fully
implementable, and a practical version replaces $A$ by $F'(x_k)$ or a fixed
surrogate. Define
\begin{equation}\label{eq:balancing}
  k_{\mathrm{bal}}=\min\Bigl\{k:\ \norm{\xsk-\xsk[j]}\le 4\,\Phi(j)
  \ \text{ for all } j>k\Bigr\}.
\end{equation}

\begin{theorem}[Oracle inequality]\label{thm:balancing}
Let \cref{ass:M} hold and consider the linearized iteration. On the event
$\mathcal{G}:=\{n_{k-1}\le\Phi(k)\ \text{for all}\ 1\le k\le k_{\max}\}$ the
balanced index satisfies the oracle inequality
\begin{equation}\label{eq:oracle}
  \norm{\xsk[k_{\mathrm{bal}}]-\xdag}\le 6\ell\,\min_{1\le k\le k_{\max}}
  \bigl(\beta_k+\Phi(k)\bigr).
\end{equation}
Consequently $k_{\mathrm{bal}}$ attains, up to a constant and without knowledge
of $\nu$ or $\psi$, the a~priori rates of \cref{thm:apriori} and \cref{thm:vsc}.
The event $\mathcal{G}$ has probability at least
$1-\sum_{k=1}^{k_{\max}}\exp(-r_{k-1}/2)$, where
$r_k:=\tr(B_k^{*}B_k)/\norm{B_k}^2$ is the stable rank of the resolvent
$B_k=(A+\alpha_kI)^{-1}$ in the discrete $L^2$ inner product.
\end{theorem}

\begin{remark}[The constant $\kappa$ in practice]\label{rem:kappa}
\Cref{thm:balancing} requires $\kappa\ge2\kappa_M\ge2$; with such $\kappa$ the
rule is more conservative and stops earlier. The experiment of
\cref{sec:numerics} uses the calibrated value $\kappa=0.3$ and $\phi_k$ from the
reference operator, so it illustrates a practical balancing heuristic rather
than verifying the theorem.
\end{remark}

The stable rank satisfies $1\le r_k\le n$. In the operating regime
$\alpha_k\gtrsim\lambda_1$ of \cref{prop:sharpvar}(i) all singular values of
$B_k$ are comparable and $r_k\asymp n$, so the probability is
$1-O(k_{\max}e^{-cn})$. When $\alpha_k$ is far below the leading eigenvalues a
few directions dominate the resolvent, $r_k$ can be much smaller than $n$, and
the exponent is correspondingly weaker. We state the bound with $r_k$ rather
than $n$ for this reason.

\begin{proof}
This is the Math\'e--Pereverzev balancing theorem
\cite[Thm.~1]{MathePereverzev2003} applied to the decomposition
\eqref{eq:linsplit}. We explain how the pieces fit.

The theorem needs two monotone bounds and a pathwise noise bound. The bias
bound $\beta_k$ is non-increasing by assumption, and $\Phi(k)$ is non-decreasing
because $\phi_k^2$ grows as $\alpha_k$ decreases. On $\mathcal{G}$,
\cref{lem:biasvar}(a) gives $\norm{\xsk-\xdag}\le\beta_k+\Phi(k)$ for every
$k$.

Under these conditions the balancing argument runs as follows. Let $k^\ast$ be
the first index with $\beta_{k^\ast}\le\Phi(k^\ast)$. For any $j>k\ge k^\ast$,
monotonicity gives $\beta_k\le\beta_{k^\ast}\le\Phi(k^\ast)\le\Phi(k)\le\Phi(j)$
and $\beta_j\le\Phi(j)$, so
\[
  \norm{\xsk-\xsk[j]}\le\norm{\xsk-\xdag}+\norm{\xsk[j]-\xdag}
  \le\beta_k+\Phi(k)+\beta_j+\Phi(j)\le4\,\Phi(j).
\]
So $k^\ast$ satisfies the test in \eqref{eq:balancing}, hence
$k_{\mathrm{bal}}\le k^\ast$. The test at $k_{\mathrm{bal}}$ with $j=k^\ast$
gives $\norm{\xsk[k_{\mathrm{bal}}]-\xsk[k^\ast]}\le4\Phi(k^\ast)$, and so
\[
  \norm{\xsk[k_{\mathrm{bal}}]-\xdag}\le4\Phi(k^\ast)+\beta_{k^\ast}+\Phi(k^\ast)
  \le6\,\Phi(k^\ast).
\]
It remains to compare $\Phi(k^\ast)$ with $\min_k(\beta_k+\Phi(k))$. For
$k\ge k^\ast$, $\beta_k+\Phi(k)\ge\Phi(k)\ge\Phi(k^\ast)$. For $k<k^\ast$, the
crossing property gives $\beta_{k^\ast-1}>\Phi(k^\ast-1)$, and since
$(\lambda+\alpha_{j-1})/(\lambda+\alpha_j)\le\alpha_{j-1}/\alpha_j\le\ell$ for
every $\lambda\ge0$, the proxy grows by at most $\ell$ per step,
$\Phi(k^\ast)\le\ell\,\Phi(k^\ast-1)$; hence
$\Phi(k^\ast)<\ell\,\beta_{k^\ast-1}\le\ell\,\beta_k\le\ell(\beta_k+\Phi(k))$. In
both cases $\Phi(k^\ast)\le\ell\min_k(\beta_k+\Phi(k))$, which with the previous
display gives \eqref{eq:oracle}. Since the right-hand side of
\eqref{eq:oracle} is, up to a constant, the quantity balanced in
\cref{thm:apriori} and \cref{thm:vsc}, $k_{\mathrm{bal}}$ attains those rates
without knowing $\nu$ or $\psi$.

For the probability of $\mathcal{G}$, fix $k$ and write
$n_{k-1}=\norm{B_{k-1}e}=(\sigma/\sqrt n)\,|B_{k-1}\xi|_h$ with
$\xi\sim\mathcal N(0,I_n)$ by \eqref{eq:noisegen}, where $|\cdot|_h$ denotes the
discrete $L^2$ norm divided by $h$, that is, the Euclidean norm. The map
$\xi\mapsto|B_{k-1}\xi|_h$ is Lipschitz with constant $\norm{B_{k-1}}$ with
respect to the Euclidean norm, and its mean is at most
$(\E|B_{k-1}\xi|_h^2)^{1/2}=\tr(B_{k-1}^{*}B_{k-1})^{1/2}$. Gaussian
concentration \cite[Thm.~5.6]{BoucheronLugosiMassart2013} gives, for $t>0$,
$\Prob\bigl(|B_{k-1}\xi|_h>(1+t)\tr(B_{k-1}^{*}B_{k-1})^{1/2}\bigr)\le
\exp\bigl(-t^2\tr(B_{k-1}^{*}B_{k-1})/(2\norm{B_{k-1}}^2)\bigr)=\exp(-t^2r_{k-1}/2)$.
Rescaling by $\sigma/\sqrt n$ and using \eqref{eq:var}, this says
$\Prob\bigl(n_{k-1}>(1+t)(\E n_{k-1}^2)^{1/2}\bigr)\le\exp(-t^2r_{k-1}/2)$. Take
$t=1$. By \eqref{eq:var2} and $\phi_{k-1}\le\phi_k$,
$2(\E n_{k-1}^2)^{1/2}\le2\kappa_M\sigma\phi_{k}\le\Phi(k)$ because
$\kappa\ge2\kappa_M$. So $\Prob(n_{k-1}>\Phi(k))\le e^{-r_{k-1}/2}$, and a union
bound over $k$ gives the claim. The exponent involves the stable rank and not
$n$: the norm of $B_{k-1}\xi$ concentrates at the scale of $\norm{B_{k-1}}$, so
only when the singular values are comparable does one recover an exponent of
order $n$.
\end{proof}

\section{Concentration of the linearized estimator}\label{app:conc}

For the linearized estimator, the error of a single realization concentrates
sharply about its median, so the expected rate is also the typical one. The
argument needs $e\mapsto\xsk$ to be globally Lipschitz, which holds for the
linearized iteration by \eqref{eq:linsplit}. For the nonlinear iteration it holds
only locally, and a localization argument, which we do not carry out, would be
needed.

\begin{theorem}[Sub-Gaussian concentration; linearized estimator]\label{thm:conc}
Let \cref{ass:M} hold. Fix $k$ and let $m_k$ be a median of
$\norm{\xsk[k+1]-\xdag}$ for the linearized iteration. Then for all $t>0$,
\begin{equation}\label{eq:conc-est}
  \Prob\bigl(\,\bigl|\norm{\xsk[k+1]-\xdag}-m_k\bigr|>t\,\bigr)
  \le 2\exp\!\Bigl(-\frac{n\,\alpha_k^{2}}{2\kappa_M^{2}\sigma^{2}}\,t^{2}\Bigr).
\end{equation}
Consequently the standard deviation of $\norm{\xsk[k+1]-\xdag}$ is at most
$\kappa_M\sigma/(\alpha_k\sqrt n)$. At the operating point $\alpha_k\asymp\sigma^{1/2}$
this is a fraction $\lesssim\kappa_M/\sqrt n$ of the target rate scale
$\sigma^{1/2}$. Since our theory provides an upper bound of order $\sigma^{1/2}$
on the mean error but not a matching lower bound, we phrase this as a statement
about the target scale rather than about the mean itself.
\end{theorem}

\begin{proof}
The proof has two steps. First we show that the error is a Lipschitz function
of the noise. Then we apply Gaussian concentration of measure.

For the first step, take two noise vectors $e$ and $e'$ and denote the
corresponding linearized iterates by $\xsk[k+1](e)$ and $\xsk[k+1](e')$. By the
identity \eqref{eq:linsplit}, both differ from the same noise-free iterate
$x_{k+1}$ by a resolvent applied to the noise, so
\[
  \norm{\xsk[k+1](e)-\xsk[k+1](e')}=\norm{(A+\alpha_kI)^{-1}(e-e')}
  \le\frac{\kappa_M}{\alpha_k}\,\norm{e-e'}
\]
by \eqref{eq:resolventL2}. By the reverse triangle inequality,
$\bigl|\norm{\xsk[k+1](e)-\xdag}-\norm{\xsk[k+1](e')-\xdag}\bigr|\le
\norm{\xsk[k+1](e)-\xsk[k+1](e')}$. So the map $e\mapsto\norm{\xsk[k+1]-\xdag}$
is Lipschitz in the discrete $L^2$ metric with constant $\kappa_M/\alpha_k$, for
every $e$ and not only locally.

For the second step, write $e=(\sigma/(h\sqrt n))\,\xi$ with
$\xi\sim\mathcal N(0,I_n)$ as in \eqref{eq:noisegen}. Since
$\norm{e}=(\sigma/\sqrt n)\,|\xi|$ with $|\cdot|$ the Euclidean norm, the map
$\xi\mapsto\norm{\xsk[k+1]-\xdag}$ is Lipschitz with constant
$\Lambda:=\kappa_M\sigma/(\alpha_k\sqrt n)$ with respect to $|\cdot|$. The
Gaussian concentration inequality \cite[Thm.~5.6]{BoucheronLugosiMassart2013}
states that a $\Lambda$-Lipschitz function $f$ of a standard Gaussian vector
satisfies $\Prob(|f(\xi)-m|>t)\le2\exp(-t^2/(2\Lambda^2))$ for a median $m$ of
$f(\xi)$. Inserting $\Lambda$ gives
$t^2/(2\Lambda^2)=n\alpha_k^2t^2/(2\kappa_M^2\sigma^2)$, which is
\eqref{eq:conc-est}.

For the standard deviation, the Gaussian Poincar\'e inequality
\cite{BoucheronLugosiMassart2013} gives $\operatorname{Var}f(\xi)\le\Lambda^2$ for
any $\Lambda$-Lipschitz $f$, so the standard deviation is at most
$\Lambda=\kappa_M\sigma/(\alpha_k\sqrt n)$. At the operating point
$\alpha_k\asymp\sigma^{1/2}$ this is $\asymp\kappa_M\sigma^{1/2}/\sqrt n$, that is,
a fraction $\lesssim\kappa_M/\sqrt n$ of the target rate scale $\sigma^{1/2}$.
\end{proof}

\begin{remark}\label{rem:concnum}
With $n=49^2$, \cref{thm:conc} bounds the standard deviation by about
$2\%\times\kappa_M$ of the target scale $\sigma^{1/2}$. The nonlinear Monte Carlo
runs of \cref{sec:numerics} give at most $0.9\%$ of the measured mean
(\cref{tab:stat,fig:box}), which suggests, but does not prove, that the nonlinear
estimator inherits this concentration.
\end{remark}

\ack
The authors thank the anonymous referees for comments that improved the paper.

\section*{References}
\bibliographystyle{iopart-num}
\bibliography{references}

\end{document}


\title[Supplementary material]{Supplementary material for\\
``Iterated Lavrentiev regularization under random noise: variational source
conditions and rates in expectation''}

\author{Ankit Singh, Hemant Sharma and Kiran Davuluri}

\maketitle

\noindent This document contains supplementary material for the article of the
same title. Section, equation, theorem, and figure numbers of the form S$m$.$n$
refer to this document; unprefixed numbers (for example Lemma~2.3 or
Section~3) refer to the main article. The material here supports, but is not
required by, the main results.
\medskip

\section{Spectral validation and the general source-condition rate}
\label{sec:advnum}

This section reports three experiments that test the finer theory of
sections 2.3 and 3 directly. They use the bounded-away source
$\xdag=1+\tfrac12\sin(\pi\xi_1)\sin(\pi\xi_2)$ on an $N=39$ grid, so that
$M_{g'(\xdag)}$ is well conditioned and fractional powers $A^{\nu}$ can be
computed through the eigendecomposition of the symmetric surrogate
$B=M_{g'(\xdag)}^{1/2}KM_{g'(\xdag)}^{1/2}$.

\paragraph{Spectrum and effective dimension} The computed eigenvalues of $A$
(\cref{fig:spectrum}, left) decay as $\lambda_i\propto i^{-0.83}$ over the whole index range, consistent
with the Weyl rate $i^{-1}$ of Lemma 2.7 once one allows for the flattening of
the discrete spectrum at high index, which lowers the fitted exponent, and for
the multiplier. The effective dimension $N_{\rm eff}(\alpha)$
(\cref{fig:spectrum}, right) grows from $2.6$ at $\alpha=1$ to $997$ at
$\alpha\approx3.9\times10^{-4}$, on $n=1521$ modes. This is the growth
$\alpha^{-1}\log(1+\alpha n)$ of (2.18): at the last value $\alpha n\approx0.59$
and the effective dimension is already two thirds of $n$. It quantifies how
many modes the method inverts as regularization is relaxed.

\paragraph{Validation of the variance bound} \Cref{fig:varvalid} compares, across
the whole parameter schedule, the Monte Carlo estimate of $\E n_k^2$ (over $200$
noise draws) with the spectral proxy $\sigma^2\phi_k^2$ of Lemma 2.4
and with the crude bound $\sigma^2/\alpha_k^2$. The Monte Carlo values agree
with $\sigma^2\phi_k^2$ to within $0.7\%$ over six orders of magnitude. This is
consistent with the two-sided bound (2.15), and much tighter than that
bound, whose constant is $\kappa_M^2=2.25$ here: the singular values of the
resolvent are close to $(\lambda_i+\alpha_k)^{-1}$ although $A$ is not
symmetric in $L^2$. The crude bound is sharp for large $\alpha_k$ and
overestimates the variance by a factor $7.3$ at the smallest $\alpha_k$, the
saturation predicted by Proposition A.1(ii) and Remark A.2.

\paragraph{The general source-condition rate} Finally we test the prediction
$\E\norm{\cdot}\sim\sigma^{\nu/(\nu+1)}$ of Remark 3.2. We construct initial
guesses with $e_0=A^{\nu}v$ for $\nu\in\{0.5,0.75,1\}$, each normalized to
$\norm{e_0}=0.05$, and run the discrepancy-stopped method over $M=50$ draws per
noise level. \Cref{tab:nu} and \cref{fig:nusweep} report fitted exponents
$0.394,0.440,0.459$. They increase monotonically with $\nu$ and are ordered as
the predicted $0.333,0.429,0.500$, but they are not equal to them: the fitted
exponent lies above the prediction for $\nu=0.5$ and below it for $\nu=1$. As
in section 5.2, these are finite-range fits over a geometric parameter
grid, and we read them as showing the predicted monotone dependence on
smoothness rather than as a quantitative confirmation of each exponent.

\begin{table}[ht]
\centering
\caption{General source condition $e_0=A^{\nu}v$ ($N=39$, $M=50$): fitted exponent
$\hat p$ in $\E\norm{\cdot}\sim\sigma^{\hat p}$ against the prediction
$\nu/(\nu+1)$.}
\label{tab:nu}
\begin{tabular}{cccc}
\toprule
$\nu$ & $0.5$ & $0.75$ & $1.0$\\
\midrule
predicted $\nu/(\nu+1)$ & $0.333$ & $0.429$ & $0.500$\\
fitted $\hat p$ & $0.394{\scriptstyle\pm0.003}$ & $0.440{\scriptstyle\pm0.007}$
& $0.459{\scriptstyle\pm0.008}$\\
\bottomrule
\end{tabular}
\end{table}

\begin{figure}[ht]
\centering
\includegraphics[width=\textwidth]{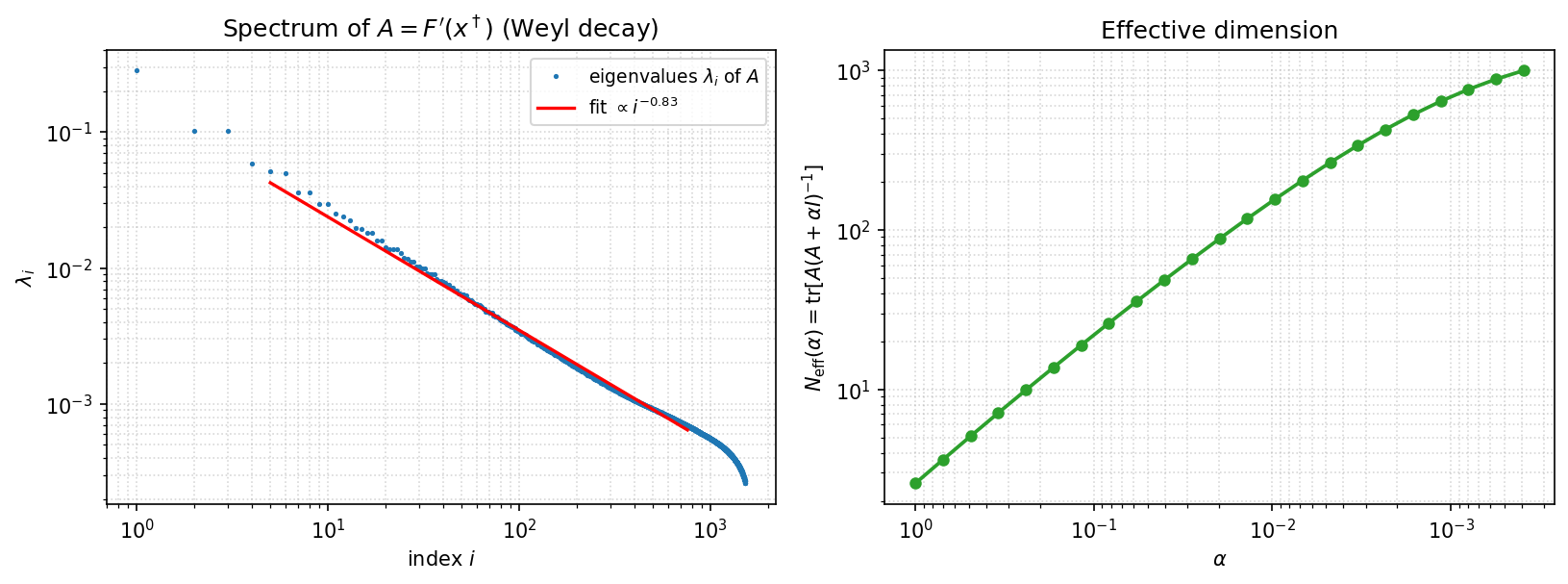}
\caption{Left: eigenvalues of $A=F'(\xdag)$ decay at the Weyl rate
(fit $\propto i^{-0.83}$). Right: the effective dimension
$N_{\rm eff}(\alpha)=\tr[A(A+\alpha I)^{-1}]$.}
\label{fig:spectrum}
\end{figure}

\begin{figure}[ht]
\centering
\begin{subfigure}{0.49\textwidth}\centering
\includegraphics[width=\textwidth]{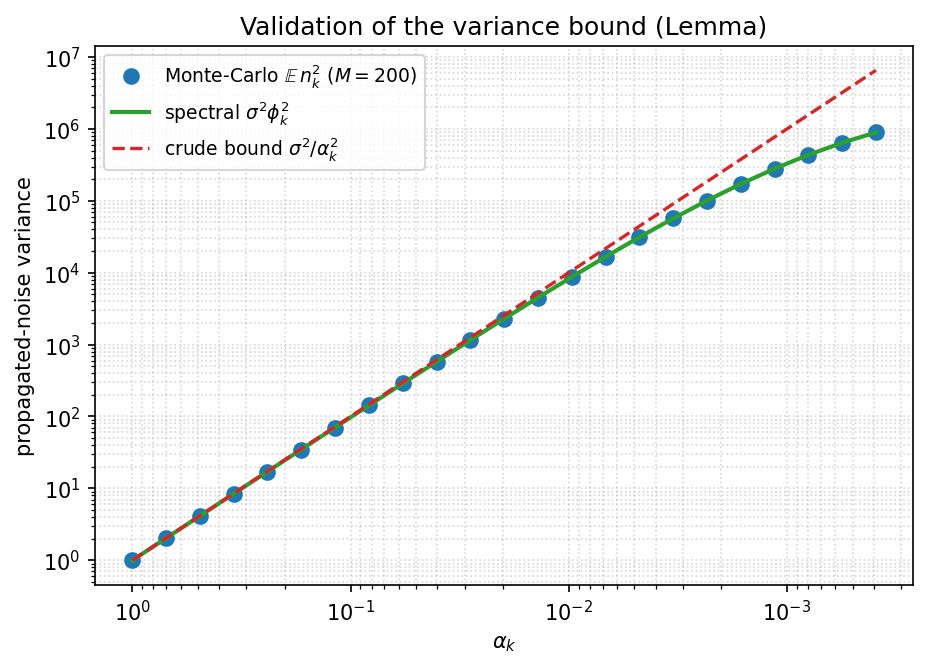}
\caption{Variance bound (Lemma 2.4).}\label{fig:varvalid}
\end{subfigure}\hfill
\begin{subfigure}{0.49\textwidth}\centering
\includegraphics[width=\textwidth]{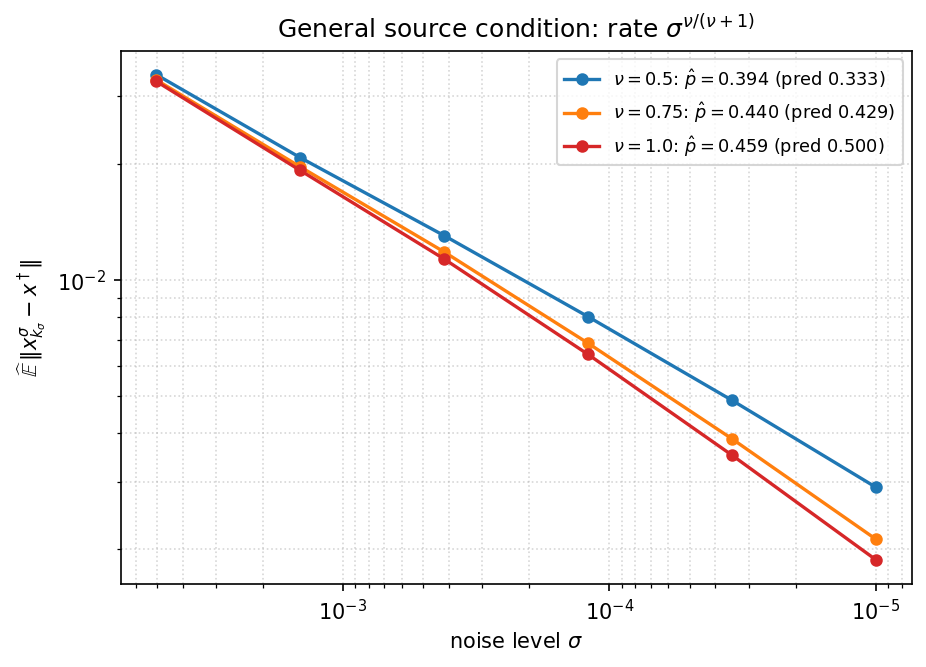}
\caption{General rate $\sigma^{\nu/(\nu+1)}$ (Remark 3.2).}\label{fig:nusweep}
\end{subfigure}
\caption{Left: Monte Carlo $\E n_k^2$ (markers) matches the spectral proxy
$\sigma^2\phi_k^2$ (green) to within $0.7\%$; the crude bound $\sigma^2/\alpha_k^2$
(red) saturates away. Right: the rate steepens with the source smoothness $\nu$, in the
order predicted by $\sigma^{\nu/(\nu+1)}$.}
\label{fig:adv}
\end{figure}

\section{Heterogeneous media, total error, and large-scale computation}
\label{sec:hetero}

The analysis so far fixes the medium and the grid. We now do three things. (i) We
extend it to a heterogeneous conductivity, making the constants explicit in the
contrast. (ii) We prove a total-error bound that couples the statistical rate
with the discretization error and identifies the optimal mesh-noise scaling.
(iii) We demonstrate the method on a large-scale heterogeneous reconstruction and
against a Tikhonov competitor.

\subsection{Heterogeneous forward operator}

Replace $-\Delta$ by $\mathcal{L}_a:=-\nabla\!\cdot(a\nabla\,\cdot\,)$ with
$0<a_-\le a(\xi)\le a_+$, which models a heterogeneous medium (for example
subsurface transport or non-destructive testing). Then $\mathcal{L}_a$ is
self-adjoint and positive with
$\lambda_1(\mathcal{L}_a)\ge a_-\lambda_1(\Lap)=2\pi^2a_-$. Its inverse
$K_a:=\mathcal{L}_a^{-1}$ has a non-negative Green's function and
$\norm{K_a}\le(2\pi^2a_-)^{-1}$. The forward operator is $F(x)=K_a\,g(x)$ with
$g(s)=s+s^3$, so $g'\ge1$ and $M_{g'(x)}$ stays uniformly positive.

\begin{proposition}[Heterogeneous verification]\label{prop:hetero}
On $B_r^{\infty}(\xdag)$ with $\rho=\norm{\xdag}_\infty+r$, the operator $F=K_ag$
satisfies Proposition 2.1 with constants
\begin{equation}\label{eq:heteroconst}
  L=\frac{3\rho}{\pi^2 a_-},\qquad a_0=\frac{1+3\rho^2}{2\pi^2 a_-};
\end{equation}
in particular $F'(x)=K_aM_{g'(x)}$ is positive-semidefinite in the energy inner
product $\inner{u}{w}_{E_a}=\inner{u}{K_a^{-1}w}$, and the whole analysis of
sections 2.2 and 3 apply as they stand with $\lambda_1$ replaced by
$2\pi^2a_-$.
\end{proposition}

\begin{proof}
The derivative is $F'(x)=K_aM_{g'(x)}$ exactly as in Proposition 2.1, with $K_a$
in place of $K$. We verify the three properties in turn.

Positivity. For $u\in X$, $\inner{F'(x)u}{u}_{E_a}=\inner{K_aM_{g'(x)}u}{K_a^{-1}u}
=\inner{M_{g'(x)}u}{u}=\int_\Om g'(x)\,u^2\ge0$, because $g'(s)=1+3s^2\ge1$. The
same computation shows that $F'(x)$ is self-adjoint in
$\inner{\cdot}{\cdot}_{E_a}$, so property (a) of Proposition 2.1 holds with the
same resolvent bounds.

Lipschitz bound. Since $g'(x)-g'(w)=3(x^2-w^2)=3(x+w)(x-w)$ and
$\norm{x}_\infty,\norm{w}_\infty\le\rho$ on the ball, we have
$\norm{g'(x)-g'(w)}_\infty\le6\rho\,\norm{x-w}_\infty$. Multiplication by a
bounded function has operator norm equal to its sup norm, so
\[
\norm{F'(x)-F'(w)}\le\norm{K_a}\,\norm{g'(x)-g'(w)}_\infty\le
(2\pi^2a_-)^{-1}\,6\rho\,\norm{x-w}_\infty,
\]
where $\norm{K_a}=1/\lambda_1(\mathcal{L}_a)\le(2\pi^2a_-)^{-1}$ follows from
$\lambda_1(\mathcal{L}_a)\ge a_-\lambda_1(\Lap)$, which in turn follows by
comparing the quadratic forms $\int a|\nabla u|^2\ge a_-\int|\nabla u|^2$. This
gives $L=3\rho/(\pi^2a_-)$.

Scaling. Similarly $\norm{g'(x)}_\infty\le1+3\rho^2$ on the ball, so
$\norm{F'(x)}\le\norm{K_a}\,\norm{g'(x)}_\infty\le(2\pi^2a_-)^{-1}(1+3\rho^2)=a_0$.

Since sections 2.2 and 3 use only these three properties, with
$\lambda_1$ entering only through the resolvent and norm bounds, the analysis
there applies with $\lambda_1$ replaced by $2\pi^2a_-$.
\end{proof}

The constants \eqref{eq:heteroconst} degrade like $a_-^{-1}$. Low-conductivity
regions amplify the ill-conditioning, as one expects physically.

\subsection{A total-error bound and the optimal mesh-noise coupling}

On a quasi-uniform mesh of width $h$ ($n\asymp h^{-2}$ degrees of freedom in
two dimensions), let $F_h$ be the discrete forward operator and $\Pi_h$ the
mesh projection. Because the deterministic theory of \cite{MahaleSingh2022} is
applied in the energy norm (Proposition 3.5), the consistency estimate we
need is in the energy norm:
\begin{equation}\label{eq:consistency}
  \norm{(F-F_h\Pi_h)\,\xdag}_E\le C_F\,h^{s_E},
\end{equation}
together with the $L^2$ projection error
$\norm{\Pi_h\xdag-\xdag}\le C_\Pi h^{s'}$. For the five-point scheme and a
smooth state the energy consistency order is the same as the $L^2$ order,
$s_E=2$. To see this, write $e:=\Pi_hu-u_h$ with $u=Kg(\xdag)$ the continuum
state and $u_h=S^{-1}\Pi_hg(\xdag)$ the discrete one, so that
$Se=S\Pi_hu-\Pi_h(-\Delta u)=:\tau$ is the truncation error, of size
$\norm{\tau}\le Ch^2$ for $u\in C^4$. Then
$\norm{e}_E^2=\inner{e}{Se}=\inner{e}{\tau}\le\norm{e}\,\norm{\tau}\le
\lambda_{\min}(S)^{-1/2}\norm{e}_E\norm{\tau}$, hence
$\norm{e}_E\le\lambda_{\min}(S)^{-1/2}\norm{\tau}\le C'h^2$. We confirmed this
numerically on a smooth test state: from $N=15$ to $N=255$ both $\norm{e}$ and
$\norm{e}_E$ decrease at the rate $2.00$ per halving of $h$. We take $s'=2$ for
the nodal projection, state \eqref{eq:consistency} as a hypothesis for general
media, and give the resulting rate in terms of $s_E$; the value $s_E=2$ is the
one used in \cref{cor:coupling}.

\begin{theorem}[Total error]\label{thm:total}
Let $x_h^{\sigma}$ be the truncated estimator of section 3 computed on the
mesh from data $\ys=F(\xdag)+e$, with the discrepancy threshold
$\tau(\sigma_E+C_Fh^{s_E})$, where $\sigma_E\asymp\sigma/h$ is the energy noise
level (2.10) and $s_E$ the energy consistency order of
\eqref{eq:consistency}. Assume \eqref{eq:consistency}, the $L^2$ projection error
above, and that the discrete problem satisfies the order-one source condition for
$\Pi_h\xdag$ and the smallness conditions of \cite{MahaleSingh2022} with constants
uniform in $h$. Then
\begin{equation}\label{eq:total}
  \Bigl(\E\,\norm{x_h^{\sigma}-\xdag}^2\Bigr)^{1/2}
  \le C_1\bigl(\sigma_E+C_Fh^{s_E}\bigr)^{1/2}+C_2\,h^{s'}
  +R\,e^{-\frac14(\beta-1)^2\rho(S)},
\end{equation}
the tail entering as a square root of the bad-event probability, as in
Proposition 3.5. The exponent is $\tfrac12$ because the nonlinear theory
is available only for the order-one source condition; we do not state a
general-$\nu$ total error.
\end{theorem}

The uniformity of the discrete source condition in $h$ is a genuine hypothesis
and is stated as such; preservation of a source condition under discretization
is not automatic. In practice the algorithm is fed a single level. Under the
coupling of \cref{cor:coupling} the discretization term $C_Fh^{s_E}$ is of the
same order as $\sigma_E$, so feeding $\sigma_E$ with a threshold constant $\tau$
enlarged by the factor $1+C_Fh^{s_E}/\sigma_E=O(1)$ is equivalent. Feeding the
$L^2$ level $\sigma$ instead, as the experiments do, under-regularizes on a fixed
mesh once $\sigma\ll C_Fh^{s_E}$, and this is the discretization floor visible
in \cref{fig:coupling}.

\begin{proof}
Split the error into a discretization part and a reconstruction part on the
mesh:
\[
  \norm{x_h^{\sigma}-\xdag}\le\norm{x_h^{\sigma}-\Pi_h\xdag}+\norm{\Pi_h\xdag-\xdag}.
\]
The second term is the projection error, bounded by $C_\Pi h^{s'}$ by
assumption. It is deterministic and gives the term $C_2h^{s'}$ in
\eqref{eq:total}.

For the first term, view the computation on the mesh as an inverse problem for
the discrete operator $F_h$ with exact solution $\Pi_h\xdag$. The data actually
used are
\[
  \ys=F(\xdag)+e=F_h\Pi_h\xdag+\eta+e,\qquad \eta:=(F-F_h\Pi_h)\xdag,
\]
so relative to the discrete problem the data error is $\eta+e$, and we measure it
in the energy norm, as Proposition 3.5 requires. The term $\eta$ is
deterministic, with $\norm{\eta}_E\le C_Fh^{s_E}$ by \eqref{eq:consistency}. The
term $e$ is Gaussian. We repeat the argument of Proposition 3.5 in the
energy norm. On the event $\mathcal{E}=\{\norm{e}_E\le\beta\sigma_E\}$ the total
data error satisfies
$\norm{\eta+e}_E\le C_Fh^{s_E}+\beta\sigma_E\le\beta(\sigma_E+C_Fh^{s_E})$, so the
deterministic theory applies with the effective energy noise level
$\delta_{h,E}:=\beta(\sigma_E+C_Fh^{s_E})$. Under the order-one source condition
the stated parameter choice with this level gives
$\norm{x_h^\sigma-\Pi_h\xdag}_E\le C\,\delta_{h,E}^{1/2}$ on $\mathcal{E}$, with
constants uniform in $h$ by the hypothesis on the discrete source condition, and
hence $\norm{x_h^\sigma-\Pi_h\xdag}\le c_0^{-1/2}C\,\delta_{h,E}^{1/2}$ in $L^2$
by (2.3). On $\mathcal{E}^c$ the truncation gives
$\norm{x_h^\sigma-\Pi_h\xdag}\le R$, and (3.3) bounds the probability
of $\mathcal{E}^c$ by $e^{-\frac12(\beta-1)^2\rho(S)}$. Combining, with the tail
entering as a square root,
\[
  \bigl(\E\norm{x_h^{\sigma}-\Pi_h\xdag}^2\bigr)^{1/2}\le
  C_1(\sigma_E+C_Fh^{s_E})^{1/2}+R\,e^{-\frac14(\beta-1)^2\rho(S)}.
\]
Adding the projection error gives \eqref{eq:total}.
\end{proof}

\begin{corollary}[Optimal coupling]\label{cor:coupling}
For the five-point scheme and a smooth solution, $s_E=s'=2$, so
$\sigma_E\asymp\sigma/h$ and $C_Fh^{s_E}\asymp h^2$. The two data-error terms
balance when $\sigma/h\asymp h^2$, that is at $h\asymp\sigma^{1/3}$ (hence
$n\asymp\sigma^{-2/3}$), giving $\sigma_E\asymp C_Fh^{s_E}\asymp\sigma^{2/3}$,
while the projection error is $h^{s'}\asymp\sigma^{2/3}$, of lower order than
the rate. Then
\begin{equation}\label{eq:coupledrate}
  \Bigl(\E\,\norm{x_h^{\sigma}-\xdag}^2\Bigr)^{1/2}
  =\mathcal{O}\bigl(\sigma_E^{1/2}\bigr)=\mathcal{O}\bigl(\sigma^{1/3}\bigr),
\end{equation}
with no discretization floor. The exponent $\tfrac13$, rather than $\tfrac12$, is
the price of the energy noise factor $\sigma_E\asymp\sigma/h$: the mesh cannot be
refined independently of $\sigma$ without inflating the energy noise level. The
concentration tail becomes $e^{-c\rho(S)}=e^{-c\,\sigma^{-2/3}}$,
super-polynomially small, so it vanishes under the coupling. The naive coupling
$h\asymp\sigma^{1/2}$ used in the experiment below gives
$\sigma_E\asymp\sigma^{1/2}$ and hence the slower rate $\sigma^{1/4}$. If one instead measures success in the energy norm,
or the model supplies noise at level $\sigma$ in the energy norm directly, the
$L^2$ noise factor disappears and the balance improves; we state the $L^2$/energy
distinction rather than hide it. The large-scale experiment of
\cref{sec:largescale} implements the naive $L^2$ coupling $h\asymp\sigma^{1/2}$
and shows only that refining the mesh with $\sigma$ lowers the discretization
floor, as discussed there.
\end{corollary}

\Cref{cor:coupling} answers the practical question of how fine a mesh a given
noise level calls for, and it is what turns the joint $(\sigma,n)$ statement of
Proposition 3.5 into a rate in $\sigma$ alone. It is tested directly in
\cref{sec:largescale}.

\subsection{Large-scale experiments}\label{sec:largescale}

\paragraph{A heterogeneous reconstruction at scale} We take a log-normal
conductivity of contrast $a_+/a_-=40$ and a four-plume source phantom on a
$192\times192$ grid ($n=36{,}864$ unknowns), with full-field data at
$\sigma=10^{-4}$ and an uninformed constant initial guess. Each iteration needs one
sparse factorization and three solves with it. The discrepancy rule stops at $k_\delta=18$ ($56$ solves,
$3.0$\,s), with relative error $0.15$. \Cref{fig:hetero} shows the medium, the
true source, the reconstruction, and the error. All four plumes are recovered,
with the expected smoothing of their peaks. The error tracks the low-conductivity
channels, exactly where \cref{prop:hetero} predicts amplified ill-conditioning.

\begin{figure}[ht]
\centering
\includegraphics[width=\textwidth]{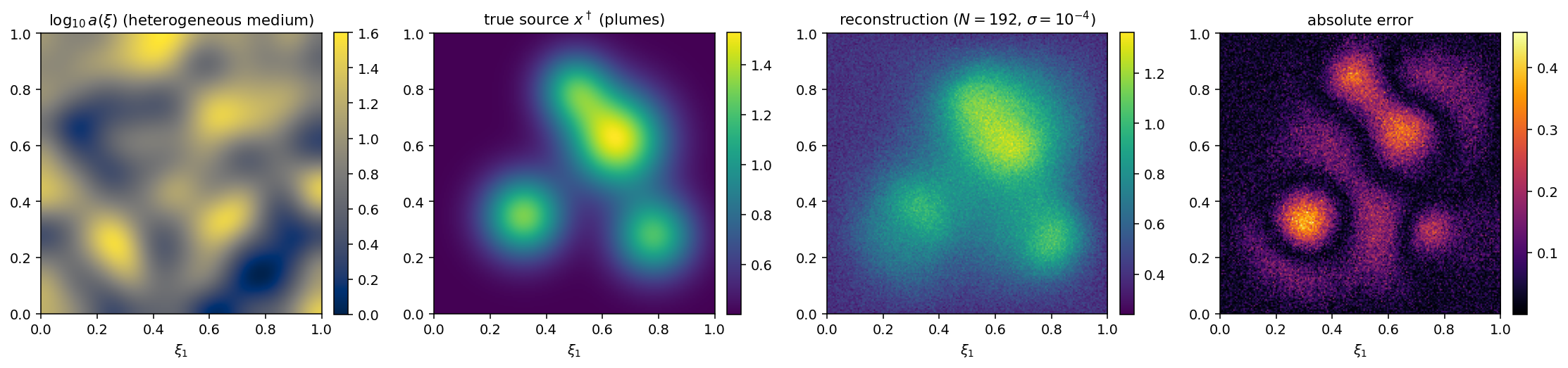}
\caption{Large-scale heterogeneous inverse source problem ($N=192$, $n=36{,}864$).
From left: log-conductivity $\log_{10}a$, true plume source, iterated Lavrentiev
reconstruction, and absolute error. The error concentrates over low-conductivity
regions, consistent with the $a_-^{-1}$ scaling of the constants in
\cref{prop:hetero}.}
\label{fig:hetero}
\end{figure}

\paragraph{Mesh-noise coupling} We test the qualitative content of
\cref{thm:total,cor:coupling} in an inverse-crime-free set-up. ``Continuous''
data are generated on a fine $224\times224$ reference grid and inverted on
coarser working grids, with the error measured on a common $128\times128$ grid.
\Cref{fig:coupling} contrasts a fixed coarse mesh ($N=24$) with a coupled choice
$N(\sigma)=\mathrm{round}(0.45\,\sigma^{-1/2})$, that is, $h\asymp\sigma^{1/2}$,
clipped to $24\le N\le128$; the coupling is therefore active only for
$3\times10^{-4}\gtrsim\sigma\gtrsim3\times10^{-5}$, and the three smallest noise
levels use $N=128$. This is the naive
$L^2$ coupling that ignores the energy noise factor of \cref{cor:coupling}; we
use it because it is the one implemented in the experiment, and we read the
result accordingly. The fixed mesh follows the rate at first but then flattens to
a floor $\approx3.4\times10^{-2}\approx h$ for $\sigma\lesssim3\times10^{-6}$, and
its fitted exponent drops to $0.23$. This is the discretization floor of
\cref{thm:total}. The coupled mesh stays below this floor and reaches
$8.7\times10^{-3}$ at $\sigma=10^{-6}$, with fitted exponent $0.39$. Refining the
mesh with $\sigma$ lowers the floor, which is the qualitative prediction of
\cref{cor:coupling}; because the mesh is capped at $N=128$, the curve would
eventually meet the floor of that mesh. The fitted
exponent $0.39$ measures the $L^2$ error under the naive coupling; it lies below
the $\tfrac12$ this coupling would give in the absence of the energy noise
factor and above the $\tfrac14$ that \cref{thm:total} predicts for it once
$\sigma_E\asymp\sigma/h$ is accounted for; the optimal coupling
$h\asymp\sigma^{1/3}$ of \cref{cor:coupling} would give $\tfrac13$. We therefore read the figure as showing that refining the
mesh with $\sigma$ lowers the discretization floor, the qualitative content of
\cref{thm:total}, and not as a measurement of a specific rate exponent.

\begin{figure}[ht]
\centering
\includegraphics[width=0.6\textwidth]{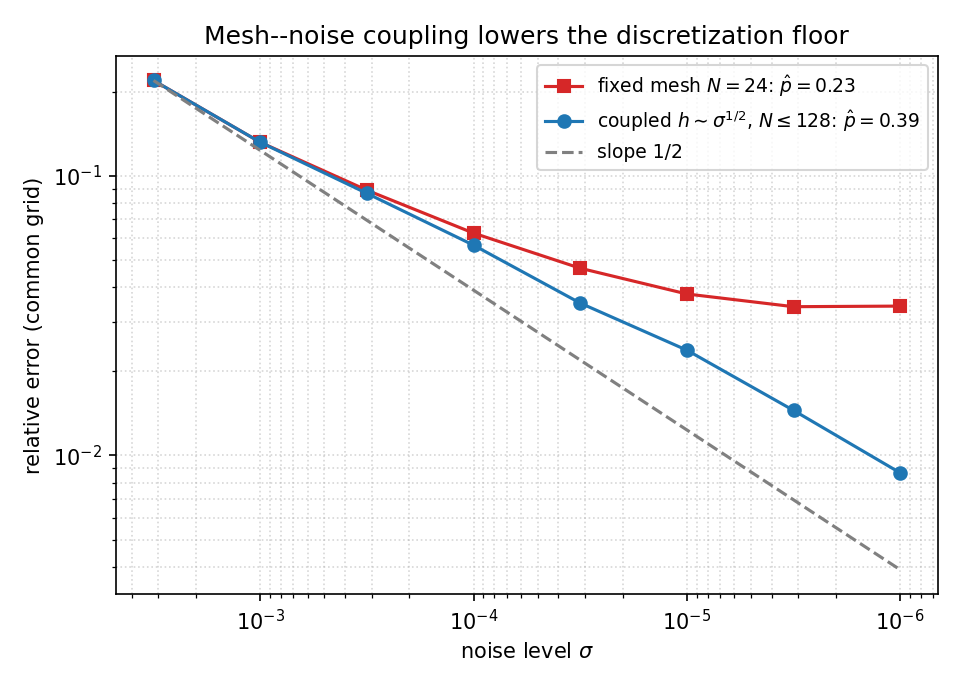}
\caption{Mesh-noise coupling (inverse-crime-free; error on a common grid). The
fixed coarse mesh (red) plateaus at the discretization floor predicted by
\cref{thm:total}. The coupled mesh $h\asymp\sigma^{1/2}$, clipped to $N\le128$
(blue), stays below that floor over the tested range; its fitted exponent
($0.39$) is not read as a rate (see \cref{cor:coupling} for the energy-norm
correction).}
\label{fig:coupling}
\end{figure}

\paragraph{Comparison with Tikhonov--Gauss--Newton} On a $112\times112$
heterogeneous problem at $\sigma=2\times10^{-4}$ we compare the iterated
Lavrentiev method with an iteratively regularized Gauss--Newton (Tikhonov)
scheme, whose normal equations $(F'^{*}F'+\alpha I)$ need the adjoint and an
inner conjugate-gradient solve. \Cref{tab:compare} reports the outcome. Gauss--Newton
reaches a somewhat smaller error ($0.122$ against $0.155$). Lavrentiev needs an
order of magnitude fewer applications of the linear solver ($50$ against $525$)
and never forms the adjoint or the normal equations. It has no conditioning
advantage on this problem, however: at the stopping level
$\alpha\approx3\times10^{-3}$ the Lavrentiev matrix
$\mathrm{diag}(g'(x))+\alpha S_a$, which inherits the stiffness of $S_a$, has
condition number $3.9\times10^{3}$, whereas the normal operator $F'^{*}F'+\alpha I$
has condition number $1.9$, because $\norm{F'}^2\approx(8/68)^2$ is small against
$\alpha$; this is why the inner conjugate-gradient solves converge in a few
iterations. The wall-clock times on this problem are similar, $0.68$\,s against
$0.60$\,s, and the next paragraph explains why. The advantage of the
adjoint-free scheme is therefore in the absence of an adjoint and in the count
of operator applications, and it turns into a time advantage only when a
factorization cannot be reused, when the adjoint is expensive, or when the
forward map is not self-adjoint.

\paragraph{What is counted, and what is not matched} So that the comparison
can be judged, here is what the numbers mean. A linear solve is one application
of a sparse direct solver: for Lavrentiev, one solve with the factorization of
$\mathrm{diag}(g'(x))+\alpha S_a$ or of $S_a$; for Gauss--Newton, one solve with
the factorization of $S_a$. Each Lavrentiev iteration costs three such solves
(one forward evaluation $F(x)$, one resolvent solve for the discrepancy
functional, one for the update) and, in addition, one new sparse LU
factorization, because the matrix $\mathrm{diag}(g'(x))+\alpha S_a$ changes with
the iterate and with $\alpha$. The pass at which the rule fires costs two
solves and no update, so stopping at index $k_\delta$ costs $3k_\delta+2$
solves and $k_\delta+1$ factorizations: $56=3\times18+2$ for the run in the
previous paragraph, and $50=3\times16+2$ with $17$ factorizations for the
comparison here. The Gauss--Newton scheme counts, per outer iteration, one
forward evaluation and one application of $S_a^{-1}$ for the right-hand side,
plus two applications of $S_a^{-1}$ per conjugate-gradient iteration; the
factorization of $S_a$ is computed once and reused throughout, and the inner
solve runs to a relative residual of $10^{-6}$ with at most $400$ iterations,
unpreconditioned. Both methods test the discrepancy before updating, so each
returns the first iterate that satisfies its rule. So Lavrentiev spends its
time in $17$ factorizations and Gauss--Newton in $525$ cheap triangular solves,
which is why the times are similar although the solve counts differ tenfold. Neither
method is preconditioned. The stopping rules are not identical: Lavrentiev uses
the resolvent discrepancy (1.5), Gauss--Newton the plain residual
discrepancy $\norm{F(x)-\ys}\le\tau\sigma$, both with $\tau=1.4$ and fed only
$\sigma$. The noise in all experiments of this section is Gaussian with
$\E\norm{e}^2=\sigma^2$, generated as in (2.9): one draw for the
large-scale and comparison runs, four draws per level for the coupling study.
We do not claim more than the table shows: a somewhat larger error at a tenth
of the solver applications and similar wall-clock time, for one problem, with
the caveats above.

\begin{table}[ht]
\centering
\caption{Iterated Lavrentiev against Tikhonov--Gauss--Newton on a $112\times112$
heterogeneous inverse source problem ($\sigma=2\times10^{-4}$).}
\label{tab:compare}
\small
\begin{tabular}{lcccc}
\toprule
method & adjoint? & linear solves & rel.\ error & subproblem cond.\ at stop \\
\midrule
iterated Lavrentiev (this work) & no  & $50$  & $0.155$ & $3.9\times10^{3}$ \\
Tikhonov--Gauss--Newton         & yes & $525$ & $0.122$ & $1.9$ \\
\bottomrule
\end{tabular}
\end{table}

\section*{References}
\bibliographystyle{iopart-num}
\bibliography{references}